\documentclass[12pt]{article}

\usepackage{amsmath}
\usepackage{amssymb}
\usepackage{amsthm}
\usepackage{amsfonts}
\usepackage{latexsym}
\usepackage{cite}
\usepackage{psfrag}
\usepackage{epsfig}
\usepackage{graphicx}
\usepackage{mathrsfs}
\usepackage{url}
\usepackage{color}
\usepackage{cancel}
\usepackage{eufrak}
\usepackage{multirow}

\usepackage[colorlinks=true]{hyperref}
\hypersetup{urlcolor=blue, citecolor=red}

\makeatletter
\@addtoreset{equation}{section}
\makeatother

\newtheorem{thm}{Theorem}[section]
\newtheorem{lem}[thm]{Lemma}
\newtheorem{corollary}[thm]{Corollary}

\theoremstyle{definition}
\newtheorem{definition}[thm]{Definition}
\newtheorem{notation}[thm]{Notation}
\newtheorem{rmk}[thm]{Remark}

\def\CC{\mathcal{C}}
\def\I{\mathcal{I}}
\def\L{\mathcal{L}}
\def\O{\mathcal{O}}
\def\P{\mathcal{P}}
\def\R{\mathcal{R}}
\def\T{\mathcal{T}}

\def\C{\mathscr{C}}
\def\Ll{\mathscr{L}}

\def\N{\mathscr{N}}
\def\U{\mathscr{U}}
\def\Ps{\mathscr{P}}

\def\PG{\mathrm{PG}}
\def\RC{\mathrm{RC}}
\def\RA{\mathrm{RA}}
\def\Tr{\mathrm{T}}
\def\IC{\mathrm{IC}}
\def\IA{\mathrm{IA}}
\def\UG{\mathrm{U\Gamma}}
\def\UnG{\mathrm{Un\Gamma}}
\def\EG{\mathrm{E\Gamma}}
\def\EnG{\mathrm{En\Gamma}}
\def\Ar{\mathrm{A}}
\def\EA{\mathrm{EA}}

\def\MM{\mathbf{M}}
\def\Pf{\mathbf{P}}

\def\F{\mathbb{F}}

\def\Nb{\mathbb{N}}

\def\A{\mathfrak{A}}
\def\Lk{\mathfrak{L}}
\def\Pk{\mathfrak{P}}
\def\pk{\mathfrak{p}}

\def\t{\text}
\def\db{\displaybreak[3]}
\def\dbn{\displaybreak[3]\notag}
\def\nt{\notag}

\begin{document}
\title
{Twisted cubic and plane-line incidence matrix in $\PG(3,q)$
\date{}
}
\maketitle
\begin{center}
{\sc Alexander A. Davydov
\footnote{A.A. Davydov ORCID \url{https://orcid.org/0000-0002-5827-4560}}
}\\
{\sc\small Institute for Information Transmission Problems (Kharkevich institute)}\\
 {\sc\small Russian Academy of Sciences}\\
 {\sc\small Moscow, 127051, Russian Federation}\\\emph {E-mail address:} adav@iitp.ru\medskip\\
 {\sc Stefano Marcugini
 \footnote{S. Marcugini ORCID \url{https://orcid.org/0000-0002-7961-0260}},
 Fernanda Pambianco
 \footnote{F. Pambianco ORCID \url{https://orcid.org/0000-0001-5476-5365}}
 }\\
 {\sc\small Department of  Mathematics  and Computer Science,  Perugia University,}\\
 {\sc\small Perugia, 06123, Italy}\\
 \emph{E-mail address:} \{stefano.marcugini, fernanda.pambianco\}@unipg.it
\end{center}

\textbf{Abstract.}
We consider the structure of the plane-line incidence matrix of the projective space $\PG(3,q)$ with respect to the orbits of planes and lines under the stabilizer group of the twisted cubic.  Structures of submatrices with incidences between a union of line orbits and an orbit of planes are investigated. For the unions consisting of two or three line orbits, the original submatrices are split into new ones, in which the incidences  are also considered.
For each submatrix (apart from the ones corresponding to a special type of lines), the numbers of lines in every plane and planes through every line are obtained.
This corresponds to the numbers of ones in columns and rows of the submatrices.



\section{Introduction}
Let $\F_{q}$ be the Galois field with $q$ elements, $\F_{q}^*=\F_{q}\setminus\{0\}$, $\F_q^+=\F_q\cup\{\infty\}$.
Let $\PG(N,q)$ be the $N$-dimensional projective space over $\F_q$.
 For an introduction to projective spaces over finite fields see \cite{Hirs_PGFF,HirsStor-2001,HirsThas-2015}.

An $n$-arc in  $\PG(N,q)$, with $n\ge N + 1\ge3$, is a
set of $n$ points such that no $N +1$ points belong to
the same hyperplane of $\PG(N,q)$. An $n$-arc is complete if it is not contained in an $(n+1)$-arc, see \cite{BallLavrauw,HirsStor-2001,HirsThas-2015} and the references therein.

In $\PG(N,q)$, $2\le N\le q-2$, a normal rational curve is any $(q+1)$-arc projectively equivalent to the arc
$\{(t^N,t^{N-1},\ldots,t^2,t,1):t\in \F_q\}\cup \{(1,0,\ldots,0)\}$. In $\PG(3,q)$, the normal rational curve is called a  \emph{twisted cubic} \cite{Hirs_PG3q,HirsThas-2015}.

Twisted cubics in $\PG(3,q)$ and its connections with a number of other objects have been widely studied; see  \cite{Hirs_PG3q}, the references therein, and  \cite{BDMP-TwCub,BonPolvTwCub,BrHirsTwCub,CLPolvT_Spr,CasseGlynn82,CasseGlynn84,CosHirsStTwCub,GiulVincTwCub,LunarPolv,HirsStor-2001,HirsThas-2015,ZanZuan2010}.  In \cite{Hirs_PG3q}, the orbits of planes, lines and points in $\PG(3,q)$ under the group $G_q$ of the projectivities fixing the twisted cubic are considered. In \cite{DMP_OrbLine}, the unions of line orbits considered in \cite{Hirs_PG3q} are investigated in detail and split in separate orbits. In \cite{BDMP-TwCub}, the structure of the \emph{point-plane} incidence matrix of $\PG(3,q)$ using orbits under $G_q$ is described.

\emph{In this paper}, we consider the structure of the \emph{plane-line} incidence matrix of $\PG(3,q)$ with respect to $G_q$.
We use the partitions of planes and lines into orbits and unions of orbits under $G_q$ described in \cite{DMP_OrbLine,Hirs_PG3q}. We research the structures of the submatrices of incidences between an orbit of planes and a union of line orbits.
 For the unions consisting of two or three line orbits, the original submatrices are split into new ones, in which the incidences  are also considered.
For each submatrix (apart from the ones corresponding to a special type of lines), the numbers of lines in every plane and planes through every line are obtained.
This corresponds to the numbers of ones in columns and rows of the submatrices.

Many submatrices considered are configurations in the sense of \cite{GroppConfig}, see Definition~\ref{def2_config} in Section \ref{subsec_incid}. Such configurations are useful in several distinct areas, in particular, to construct bipartite graph codes without the so-called 4-cycles, see e.g.  \cite{BargZem,DGMP_BipGraph,HohJust} and the references therein.

The theoretic results hold for $q\ge5$. For $q = 2,3,4$, we describe the incidence
matrices by computer search.

The results obtained increase our knowledge on the structure,  properties, and incidences of planes and lines of $\PG(3,q)$.

The paper is organized as follows. Section \ref{sec_prelimin} contains preliminaries. In Section \ref{sec_mainres}, the main results of this paper are summarized. Some useful relations are given in Section~\ref{sec:useful}.  The numbers of distinct planes in $\PG(3,q)$ through distinct lines and vice versa are obtained in Sections \ref{sec:results_q_ne0}--\ref{sec-split}.

\section{Preliminaries}\label{sec_prelimin}

\subsection{Twisted cubic}\label{subset_twis_cub}
We summarize the results on the twisted cubic of \cite{Hirs_PG3q} useful in this paper.

Let $\Pf(x_0,x_1,x_2,x_3)$ be a point of $\PG(3,q)$ with homogeneous coordinates $x_i\in\F_{q}$.
For $t\in\F_q^+$, let  $P(t)$ be a point such that
\begin{align}\label{eq2:P(t)}
  P(t)=\Pf(t^3,t^2,t,1)\text{ if }t\in\F_q;~~P(\infty)=\Pf(1,0,0,0).
\end{align}

Let $\C\subset\PG(3,q)$ be the \emph{twisted cubic} consisting of $q+1$ points $P_1,\ldots,P_{q+1}$  no four of which are coplanar.
We consider $\C$ in the canonical form
\begin{align}\label{eq2_cubic}
&\C=\{P_1,P_2,\ldots,P_{q+1}\}=\{P(t)\,|\,t\in\F_q^+\}.
\end{align}

A \emph{chord} of $\C$ is a line through a pair of real points of $\C$ or a pair of complex conjugate points. In the last  case it is an \emph{imaginary chord}. If the real points are distinct, it is a \emph{real chord}; if they coincide with each other, it is a \emph{tangent.}

Let $\boldsymbol{\pi}(c_0,c_1,c_2,c_3)$ $\subset\PG(3,q)$, be the plane with equation
\begin{align}\label{eq2_plane}
  c_0x_0+c_1x_1+c_2x_2+c_3x_3=0,~c_i\in\F_q.
\end{align}
The \emph{osculating plane} in the  point $P(t)\in\C$ is as follows:
\begin{align}\label{eq2_osc_plane}
&\pi_\t{osc}(t)=\boldsymbol{\pi}(1,-3t,3t^2,-t^3)\t{ if }t\in\F_q; ~\pi_\t{osc}(\infty)=\boldsymbol{\pi}(0,0,0,1).
\end{align}
 The $q+1$ osculating planes form the osculating developable $\Gamma$ to $\C$, that is a \emph{pencil of planes} for $q\equiv0\pmod3$ or a \emph{cubic developable} for $q\not\equiv0\pmod3$.

 An \emph{axis} of $\Gamma$ is a line of $\PG(3,q)$ which is the intersection of a pair of real planes or complex conjugate planes of $\Gamma$. In the last  case it is a \emph{generator} or an \emph{imaginary axis}. If the real planes are distinct it is a \emph{real axis}; if they coincide with each other it is a \emph{tangent} to $\C$.

For $q\not\equiv0\pmod3$, the null polarity $\A$ \cite[Sections 2.1.5, 5.3]{Hirs_PGFF}, \cite[Theorem~21.1.2]{Hirs_PG3q} is given by
\begin{align}\label{eq2_null_pol}
&\Pf(x_0,x_1,x_2,x_3)\A=\boldsymbol{\pi}(x_3,-3x_2,3x_1,-x_0).
\end{align}

\begin{notation}\label{notation_1}
In future, we consider $q\equiv\xi\pmod3$ with $\xi\in\{-1,0,1\}$. Many values depend of $\xi$ or have sense only for specific $\xi$.
We note this by remarks or by superscripts ``$(\xi)$''.
 If a value is the same for all $q$ or a property holds for all $q$, or it is not relevant, or it is clear by the context, the remarks and superscripts ``$(\xi)$'' are not used. If a value is the same for $\xi=-1,1$, then one may use ``$\ne0$''. In superscripts, instead of ``$\bullet$'', one can write ``$\mathrm{od}$'' for odd $q$ or ``$\mathrm{ev}$'' for even $q$. If a value is the same for odd and even $q$, one may omit ``$\bullet$''.

The following notation is used.
\begin{align*}
  &G_q && \t{the group of projectivities in } \PG(3,q) \t{ fixing }\C;\db  \\
  &\mathbf{Z}_n&&\t{cyclic group of order }n;\db  \\
  &\mathbf{S}_n&&\t{symmetric group of degree }n;\db  \\
&A^{tr}&&\t{the transposed matrix of }A;\db \\
&\#S&&\t{the cardinality of a set }S;\db\\
&\overline{AB}&&\t{the line through the points $A$ and }B.\db\\
&&&\t{\textbf{Types $\pi$ of planes:}}\db\\
&\Gamma\t{-plane}  &&\t{an osculating plane of }\Gamma;\db \\
&d_\C\t{-plane}&&\t{a plane containing \emph{exactly} $d$ distinct points of }\C,~d=0,2,3;\db \\
&\overline{1_\C}\t{-plane}&&\t{a plane not in $\Gamma$ containing \emph{exactly} 1 point of }\C;\db \\
&\Pk&&\t{the list of possible types $\pi$ of planes},~\Pk\triangleq\{\Gamma,2_\C,3_\C,\overline{1_\C},0_\C\};\db\\
&\pi\t{-plane}&&\t{a plane of the type }\pi\in\Pk; \db\\
&\N_\pi&&\t{the orbit of $\pi$-planes under }G_q,~\pi\in\Pk.\db\\
&&&\t{\textbf{Types $\lambda$ of lines with respect to cubic $\C$:}}\db\\
&\RC\t{-line}&&\t{a real chord  of $\C$;}\db \\
&\RA\t{-line}&&\t{a real axis of $\Gamma$ for }\xi\ne0;\db \\
&\Tr\t{-line}&&\t{a tangent to $\C$};\db \\
&\IC\t{-line}&&\t{an imaginary chord  of $\C$;}\db \\
&\IA\t{-line}&&\t{an imaginary axis of $\Gamma$ for }\xi\ne0;\db \\
&\UG&&\t{a non-tangent unisecant in a $\Gamma$-plane;}\db \\
&\t{Un$\Gamma$-line}&&\t{a unisecant not in a $\Gamma$-plane (it is always non-tangent);}\db \\
&\t{E$\Gamma$-line}&&\t{an external line in a $\Gamma$-plane (it cannot be a chord);}\db \\
&\t{En$\Gamma$-line}&&\t{an external line, other than a chord, not in a $\Gamma$-plane;}\db \\
&\Ar\t{-line}&&\t{the axis of $\Gamma$ for }\xi=0\db\\
&&&\t{(it is the single line of intersection of all the $q+1~\Gamma$-planes)};\db \\
&\EA\t{-line}&&\t{an external line meeting the axis of $\Gamma$ for }\xi=0;\db\\
&\Lk^{(\xi)}&&\t{the list of possible types $\lambda$ of lines},\db\\
&&&\Lk^{(\ne0)}\triangleq\{\RC,\RA,\Tr,\IC,\IA,\UG,\UnG,\EG,\EnG\}\t{ for }\xi\ne0,\db\\
&&&\Lk^{(0)}\triangleq\{\RC,\Tr,\IC,\UG,\UnG, \EnG,\Ar,\EA\}\t{ for }\xi=0;\db\\
&\lambda\t{-line}&&\t{a line of the type }\lambda\in\Lk^{(\xi)};\db\\
&L_\Sigma^{(\xi)}&&\t{the total number of orbits of lines in }PG(3,q);\db\\
&L_{\lambda\Sigma}^{(\xi)\bullet}&&\t{the total number of orbits of $\lambda$-lines},~\lambda\in\Lk^{(\xi)};\db\\
&\O_\lambda&&\t{the union (class) of all $L_{\lambda\Sigma}^{(\xi)\bullet}$ orbits of $\lambda$-lines under }G_q,~\lambda\in\Lk^{(\xi)}.
\end{align*}
\end{notation}

The following theorem summarizes results from \cite{Hirs_PG3q} useful in this paper.
\begin{thm}\label{th2_Hirs}
\emph{\cite[Chapter 21]{Hirs_PG3q}} The following properties of the twisted cubic $\C$ of \eqref{eq2_cubic} hold:
  \begin{align}
  &\textbf{\emph{(i)}} \t{ The group $G_q$ acts triply transitively on $\C$. Also,}\dbn\\
  &\t{\textbf{\emph{(a)}}}~~ G_q\cong PGL(2,q)~\t{ for }q\ge5;\dbn \\
 &\phantom{\t{\textbf{\emph{(a)}}}~~} G_4\cong\mathbf{S}_5\cong P\Gamma L(2,4)\cong\mathbf{Z}_2PGL(2,4),~\#G_4=2\cdot\#PGL(2,4)=120;\dbn \\
 &\phantom{\t{\textbf{\emph{(a)}}}~~} G_3\cong\mathbf{S}_4\mathbf{Z}_2^3,\hspace{4.2cm}\#G_3=8\cdot\#PGL(2,3)=192;\dbn \\
 &\phantom{\t{\textbf{\emph{(a)}}}~~} G_2\cong\mathbf{S}_3\mathbf{Z}_2^3,\hspace{4.2cm}\#G_2=8\cdot\#PGL(2,2)=48.\dbn\\
&\textbf{\emph{(b)}} \t{ The matrix $\MM$ corresponding to a projectivity of $G_q$ has the general form}\dbn\\
& \label{eq2_M} \mathbf{M}=\left[
 \begin{array}{cccc}
 a^3&a^2c&ac^2&c^3\\
 3a^2b&a^2d+2abc&bc^2+2acd&3c^2d\\
 3ab^2&b^2c+2abd&ad^2+2bcd&3cd^2\\
 b^3&b^2d&bd^2&d^3
 \end{array}
  \right],~a,b,c,d\in\F_q,\db\\
 & ad-bc\ne0.\nt
\end{align}

\textbf{\emph{(ii)}} Under $G_q$, $q\ge5$, there are the following five orbits $\N_\pi$ of planes:
\begin{align}\label{eq2_plane orbit_gen}
   &\N_1=\N_\Gamma=\{\Gamma\t{-planes}\},~~~~\#\N_\Gamma=q+1;\db\\
   &\N_{2}=\N_{2_\C}=\{2_\C\t{-planes}\}, ~\#\N_{2_\C}=q^2+q;\dbn \\
 &\N_{3}=\N_{3_\C}=\{3_\C\t{-planes}\},~  \#\N_{3_\C}=(q^3-q)/6;\dbn\\
 &\N_{4}=\N_{\overline{1_\C}}=\{\overline{1_\C}\t{-planes}\},~\#\N_{\overline{1_\C}}=(q^3-q)/2;\dbn\\
 &\N_{5}=\N_{0_\C}=\{0_\C\t{-planes}\},~\#\N_{0_\C}=(q^3-q)/3.\nt
 \end{align}

 \textbf{\emph{(iii)}} For $q\not\equiv0\pmod3$, the null polarity $\A$ \eqref{eq2_null_pol} interchanges $\C$ and $\Gamma$ and their corresponding chords and axes.

 \textbf{\emph{(iv)}} The lines of $\PG(3,q)$ can be partitioned into classes called $\O_i$ and $\O'_i$, each of which is a union of orbits under $G_q$.
  \begin{align}
  &\hspace{0.6cm}\textbf{\emph{(a)}}~ q\not\equiv0\pmod3,~ q\ge5, ~\O'_i=\O_i\A,~ \#\O'_i=\#\O_i,~i=1,\ldots,6.\dbn\\
  &\O_1=\O_\RC=\{\RC\t{-lines}\},~\O'_1=\O_\RA=\{\RA\t{-lines}\},\db\label{eq2_classes line q!=0mod3}\\
  &\#\O_\RC=\#\O_\RA=(q^2+q)/2;\dbn\\
  &\O_2=\O'_2=\O_\Tr=\{\Tr\t{-lines}\},~\#\O_\Tr=q+1;\dbn \\
  &\O_3=\O_\IC=\{\IC\t{-lines}\},~\O'_3=\O_\IA=\{\IA\t{-lines}\},~\#\O_\IC=\#\O_\IA=(q^2-q)/2;\dbn\\
  &\O_4=\O'_4=\O_\UG=\{\UG\t{-lines}\},~\#\O_\UG=q^2+q;\dbn\\
  &\O_5=\O_\UnG=\{\UnG\t{-lines}\},\O'_5=\O_\EG=\{\EG\t{-lines}\},\#\O_\UnG=\#\O_\EG=q^3-q;\dbn\\
  &\O_6=\O'_6=\O_\EnG=\{\EnG\t{-lines}\},~\#\O_\EnG=(q^2-q)(q^2-1).\nt
     \end{align}
  For $q>4$ even, the lines in the regulus complementary to that of the tangents form an orbit of size $q+1$ contained in $\O_4=\O_\UG$.
  \begin{align}
  &\textbf{\emph{(b)}}~q\equiv0\pmod3,~q>3.\dbn\\
  &\t{Classes }\O_1,\ldots,\O_6\t{ are as in \eqref{eq2_classes line q!=0mod3}};~\O_7=\O_\Ar=\{\Ar\t{-line}\},~\#\O_\Ar=1;\label{eq2_classes line q=0mod3}\\
  &\O_8=\O_\EA=\{\EA\t{-lines}\},~\#\O_\EA=(q+1)(q^2-1). \nt
     \end{align}

 \textbf{\emph{(v)}} The following properties of chords and axes hold.

 \textbf{\emph{(a)}}  For all $q$, no two chords of $\C$ meet off $\C$.

 \phantom{\textbf{\emph{(a)}}} Every point off $\C$ lies on exactly one chord of $\C$.

 \textbf{\emph{(b)}}       Let $q\not\equiv0\pmod3$.

 \phantom{\textbf{\emph{(b)}}}  No two axes of $\Gamma$ meet unless they lie in the same plane of $\Gamma$.

 \phantom{\textbf{\emph{(b)}}}  Every plane not in $\Gamma$ contains exactly one axis of $\Gamma$.

  \textbf{\emph{(vi)}} For $q>2$, the unisecants of $\C$ such that every plane through such a unisecant meets $\C$ in at most one point other than the point of contact are, for $q$ odd, the tangents, while for $q$ even, the tangents and the unisecants in the complementary regulus.
\end{thm}

The following theorem summarizes results from \cite{DMP_OrbLine} useful in this paper.
\begin{thm}\label{th2:DMP_Orb}
For line orbits under $G_q$ the following holds.
\begin{description}
  \item[(i)] The following classes of lines consist of a single orbit:\\
   $\O_1=\O_\RC=\{\RC\t{-lines}\}$,
  $\O_2=\O_\Tr=\{\Tr\t{-lines}\}$, and\\
   $\O_3=\O_\IC=\{\IC\t{-lines}\}$,  for all~$q$;\\
   $\O_4=\O_\UG=\{\UG\t{-lines}\}$, for odd $q$;\\
    $\O_5=\O_\UnG=\{\UnG\t{-lines}\}$ and $\O'_5=\O_\EG=\{\EG\t{-lines}\}$, for even $q$;\\
     $\O_1'=\O_\RA=\{\RA\t{-lines}\}$ and $\O_3'=\O_\IA=\{\IA\t{-lines}\}$,
  for $\xi\ne0$;\\
   $\O_7=\O_\Ar=\{\Ar\t{-lines}\}$, for $\xi=0$.

  \item[(ii)]
Let $q\ge8$ be even. The non-tangent unisecants in a $\Gamma$-plane \emph{(}i.e. $\UG$-lines, class $\O_4=\O_\UG$\emph{)} form two orbits of size $q+1$ and $q^2-1$. The orbit of size $q+1$  consists of the lines in the regulus complementary to that of the tangents. Also, the $(q+1)$-orbit and $(q^2-1)$-orbit can be represented in the form $\{\ell_1\varphi|\varphi\in G_q\}$ and $\{\ell_2\varphi|\varphi\in G_q\}$, respectively, where $\ell_j$ is a line such that  $\ell_1=\overline{P_0\Pf(0,1,0,0)}$, $\ell_2=\overline{P_0\Pf(0,1,1,0)}$, $P_0=\Pf(0,0,0,1)\in\C$.

  \item[(iii)] Let $q\ge5$ be odd.
 The non-tangent unisecants not in a $\Gamma$-plane \emph{(}i.e. $\UnG$-lines, class $\O_5=\O_\UnG$\emph{)} form  two orbits of size $\frac{1}{2}(q^3-q)$. These orbits can be represented in the form $\{\ell_j\varphi|\varphi\in G_q\}$, $j=1,2$, where $\ell_j$ is a line such that $\ell_1=\overline{P_0\Pf(1,0,1,0)}$,  $\ell_2=\overline{P_0\Pf(1,0,\rho,0)}$, $P_0=\Pf(0,0,0,1)\in\C$, $\rho$ is not a square.

   \item[(iv)] \looseness -1
Let $q\ge5$ be odd. Let $q\not\equiv0\pmod 3$. The external lines in a $\Gamma$-plane   \emph{(}class $\O_5'=\O_\EG$\emph{)} form two orbits of size $(q^3-q)/2$. These orbits can be represented in the form $\{\ell_j\varphi|\varphi\in G_q\}$, $j=1,2$, where $\ell_j=\pk_0\cap\pk_j$ is the intersection line of planes $\pk_0$ and $\pk_j$ such that
           $\pk_0=\boldsymbol{\pi}(1,0,0,0)=\pi_\t{\emph{osc}}(0)$, $\pk_1=\boldsymbol{\pi}(0,-3,0,-1)$,  $\pk_2=\boldsymbol{\pi}(0,-3\rho,0,-1)$, $\rho$ is not a square, cf. \eqref{eq2_plane}, \eqref{eq2_osc_plane}.

   \item[(v)]
Let $q\equiv0\pmod 3,\; q\ge9$. The external lines meeting the axis of $\Gamma$ \emph{(}i.e. $\EA$-lines, class $\O_8=\O_\EA$\emph{)} form three orbits of size $q^3-q$, $(q^2-1)/2$, $(q^2-1)/2$. The $(q^3-q)$-orbit and the two $(q^2-1)/2$-orbits can be represented in the form $\{\ell_1\varphi|\varphi\in G_q\}$ and $\{\ell_j\varphi|\varphi\in G_q\}$, $j=2,3$, respectively, where $\ell_j$ are lines such that $\ell_1=\overline{P_0^\Ar\Pf(0,0,1,1)}$,  $\ell_2=\overline{P_0^\Ar\Pf(1,0,1,0)}$, $\ell_3=\overline{P_0^\Ar\Pf(1,0,\rho,0)}$, $P_0^\Ar=\Pf(0,1,0,0)$, $\rho$ is not a square.
\end{description}
\end{thm}
\subsection{The plane-line incidence matrix  of $\PG(3,q)$}\label{subsec_incid}
The space $\PG(N,q)$ contains $\theta_{N,q}$ points and hyperplanes, and $\beta_{N,q}$ lines;
\begin{align}\label{eq1_theta_lambda}
 \theta_{N,q}=\frac{q^{N+1}-1}{q-1} ,~\beta_{N,q}=\frac{(q^{N+1}-1)(q^{N+1}-q)}{(q^2-1)(q^2-q)}\,.
\end{align}

Let  $\I^{\Pi\Lambda}$ be the $\beta_{3,q}\times\theta_{3,q}$ plane-line incidence matrix  of $\PG(3,q)$ in which columns correspond to planes, rows correspond to lines, and there is an entry  ``1'' if the corresponding line lies in the corresponding plane.  Every column and every row of $\I^{\Pi\Lambda}$ contains $\theta_{2,q}$ and $\theta_{1,q}$ ones, respectively. Thus, $\I^{\Pi\Lambda}$ is a tactical configuration \cite[Chapter 2.3]{Hirs_PGFF}, \cite[Chapter 7, Section 2]{Lidl_Nied}.
Moreover, $\I^{\Pi\Lambda}$  gives a 2-$(\theta_{3,q},\theta_{1,q},1)$ design~\cite{HandbCombDes2v_k_lamb} since there is exactly one line as the intersection of any two planes.

\begin{definition}\label{def2_config}\cite{GroppConfig}
    A configuration $(v_r,b_k)$  is an incidence structure of $v$ points and $b$ lines such that
 each line contains $k$ points, each point lies on $r$ lines, and
 two different points are connected by at most one line. If $v = b$ and, hence, $r = k$, the configuration is symmetric, denoted by $v_k$.
\end{definition}
\noindent For an introduction to the configurations see \cite{DFGMP_SymConf,GroppConfig} and the references therein.

The transposition $(\I^{\Pi\Lambda})^{tr}$ gives the $\theta_{3,q}\times\beta_{3,q}$ line-plane incidence matrix. It can be viewed as a $(v_r,b_k)$ configuration  with $v=\beta_{3,q}$, $b=\theta_{3,q}$, $r=\theta_{1,q}$, $k=\theta_{2,q}$, as there is at most one plane through two different lines.

\section{The main results}\label{sec_mainres}
Throughout the paper, we consider orbits of lines and planes under $G_q$.
\begin{notation}\label{notation_2}
In addition to Notation \ref{notation_1}, the following notation is used:
\begin{align*}
&\O_{\lambda_j}&&\t{the $j$-th orbit of the class }\O_\lambda,~ j=1,\ldots,L_{\lambda\Sigma}^{(\xi)\bullet},~\O_\lambda=\bigcup_{j=1}^{L_{\lambda\Sigma}^{(\xi)\bullet}}\O_{\lambda_j};\db\\
&\lambda_j\t{-lines}&&\lambda\t{-lines forming the $j$-th orbit $\O_{\lambda_j}$ of the class }O_{\lambda},~\lambda\in\Lk^{(\xi)};\db\\
&\Lambda_{\lambda_j,\pi}^{(\xi)\bullet}&&\t{the number of lines from an orbit $\O_{\lambda_j}$ in a $\pi$-plane};\db\\
&\Lambda_{\lambda,\pi}^{(\xi)\bullet}&&\t{the total number of $\lambda$-lines in a $\pi$-plane};\db\\
&\Pi_{\pi,\lambda_j}^{(\xi)\bullet} &&\t{the exact number of $\pi$-planes through a line of an orbit $\O_{\lambda_j}$};\db\\
&\Pi_{\pi,\lambda}^{(\xi)\bullet}&&\t{the average number of $\pi$-planes through a $\lambda$-line over all the}\db\\
&&&\t{$\lambda$-lines; if the union (class) $\O_\lambda$ consists of \emph{a single orbit} then,}\db\\
&&&\t{in fact, }\Pi_{\pi,\lambda}^{(\xi)}\t{ is \emph{the exact number} of $\pi$-planes through each $\lambda$-line};\db\\
&\I_{\pi,\lambda}^{\Pi\Lambda}&&\t{the $\#\O_\lambda\times\#\N_\pi$ submatrix of the plane-line incidence matrix $\I^{\Pi\Lambda}$}\db\\
&&&\t{with incidencies between $\pi$-planes and $\lambda$-lines;} \db\\
&\I_{\pi,\lambda_j}^{\Pi\Lambda}&&\t{the $\#\O_{\lambda_j}\times\#\N_\pi$ submatrix of $\I^{\Pi\Lambda}_{\pi,\lambda}$}\db\\
&&&\t{with incidencies between $\pi$-planes and $\lambda_j$-lines.}
 \end{align*}
\end{notation}
\begin{rmk}
If $L_{\lambda\Sigma}^{(\xi)\bullet}=1$ then $\Pi_{\pi,\lambda}^{(\xi)\bullet}$ certainly is an integer. If $\lambda$-lines form two or more orbits, i.e. $L_{\lambda\Sigma}^{(\xi)\bullet}\ge2$, then $\Pi_{\pi,\lambda}^{(\xi)\bullet}$ may be not  integer as well as integer.

 On the other end, for all pairs $(\pi,\lambda)$, we always have the same total number of $\lambda$-lines in each $\pi$-plane, i.e. $\Lambda_{\lambda,\pi}^{(\xi)\bullet}$ always is an integer,  see Lemma \ref{lem4_line&nplane}.
\end{rmk}

From now on, we consider $q\ge5$ apart from Theorem \ref{th3:q=2 3 4}.

Tables \ref{tab1} and \ref{tabUG} and Theorem \ref{th3_main_res} summarize the results of  Sections \ref{sec:useful}--\ref{sec:esults_q=0}. Theorem~\ref{th3:q=2 3 4} is obtained by an  exhaustive computer search using the symbol calculation system Magma~\cite{Magma}.

For the plane-line incidence matrix $\I^{\Pi\Lambda}$ of $\PG(3,q)$, $q\equiv\xi\pmod3$, Table~\ref{tab1}  shows the values $\Pi_{\pi,\lambda}^{(\xi)}$ (top entry) and $\Lambda_{\lambda,\pi}^{(\xi)}$ (bottom entry) for each pair $(\pi,\lambda)$, $\pi\in\Pk$,  where $\Pi_{\pi,\lambda}^{(\xi)}$  is  the exact (if $L_{\lambda\Sigma}^{(\xi)\bullet}=1$) or average (if $L_{\lambda\Sigma}^{(\xi)\bullet}\ge2$) number of $\pi$-planes through every $\lambda$-line, whereas $\Lambda_{\lambda,\pi}^{(\xi)}$ always is the exact number of $\lambda$-lines in every $\pi$-plane. In other words, $\Pi_{\pi,\lambda}^{(\xi)}$  is the exact or average number of ones in every row of the submatrix $\I^{\Pi\Lambda}_{\pi,\lambda}$  of $\I^{\Pi\Lambda}$, whereas $\Lambda_{\lambda,\pi}^{(\xi)}$ always is the exact number of ones in every column of $\I^{\Pi\Lambda}_{\pi,\lambda}$. The superscript $(\xi)$ is omitted for $\lambda\in\{\RC,\Tr,\IC,\UG,\UnG\}$ where the values $\Pi_{\pi,\lambda}^{(\xi)}$, $\Lambda_{\lambda,\pi}^{(\xi)}$ are the same for all~$q$.

\begin{table}[htbp]
\caption{Values $\Pi_{\pi,\lambda}^{(\xi)}$ (top entry) and $\Lambda_{\lambda,\pi}^{(\xi)}$ (bottom entry) for submatrices $\I_{\pi,\lambda}^{\Pi\Lambda}$ of the plane-line incidence matrix of $\PG(3,q)$, $q\equiv\xi\pmod3$, $\xi\in\{1,-1,0\}$, $q\ge5$, $\pi\in\Pk$, $\lambda\in\Lk^{(\ne0)}\cup\Lk^{(0)}$. The superscript~$(\xi)$ is omitted  if the values $\Pi_{\pi,\lambda}^{(\xi)}$ and $\Lambda_{\lambda,\pi}^{(\xi)}$ are the same for all~$q$}
$
\begin{array}
{@{}c@{~\,}ccccccc@{}}\hline
&\O_\lambda&\N_\pi\rightarrow&\Gamma\t{-}&2_\C\t{-}&3_\C\t{-}&\overline{1_\C}\t{-}\vphantom{H^{H^{H^H}}}&0_\C\t{-}\\
L_{\lambda\Sigma}^{\mathrm{od}}&\downarrow&&\,\t{planes}&\t{planes}&\t{planes}&\t{planes}&\t{planes}\\
L_{\lambda\Sigma}^{\mathrm{ev}}&\lambda\t{-lines}&&q+1&q^2+q&\frac{1}{6}(q^3-q)& \frac{1}{2}(q^3-q)&\frac{1}{3}(q^3-q)\vphantom{H_{H_H}}\\\hline

1&\RC\t{-lines}&\Pi_{\pi,\RC}&0&2&q-1&0&0\\
1&\frac{1}{2}(q^2+q)&\Lambda_{\RC,\pi}&0&1&3&0&0\\\hline

1&\Tr\t{-lines}&\Pi_{\pi,\Tr}&1&q&0&0&0\\
1&q+1&\Lambda_{\Tr,\pi}&1&1&0&0&0\\\hline

1&\IC\t{-lines}&\Pi_{\pi,\IC}&0&0&0&q+1&0\\
1&\frac{1}{2}(q^2-q)&\Lambda_{\IC,\pi}&0&0&0&1&0\\\hline

1&\UG\t{-lines}&\Pi_{\pi,\UG}&1&1&\frac{1}{2}(q-1)&\frac{1}{2}(q-1)&0\\
2&q^2+q&\Lambda_{\UG,\pi}&q&1&3&1&0\\\hline

2&\UnG\t{-lines}&\Pi_{\pi,\UnG}&0&2& \frac{1}{2}(q-2)& \frac{1}{2}q&0\\
1&q^3-q&\Lambda_{\UnG,\pi}&0&2(q-1)& 3(q-2)& q&0\\\hline

1&\RA\t{-lines}&\Pi_{\pi,\RA}^{(1)}&2&0&\frac{1}{3}(q-1)&0&\frac{2}{3}(q-1)\\
1&\frac{1}{2}(q^2+q)&\Lambda^{(1)}_{\RA,\pi}&q&0& 1& 0& 1\\\hline

1&\RA\t{-lines}&\Pi_{\pi,\RA}^{(-1)}&2&0&0& q-1&0\\
1&\frac{1}{2}(q^2+q)&\Lambda^{(-1)}_{\RA,\pi}&q&0& 0& 1& 0\\\hline

 1&\IA\t{-lines}&\Pi^{(1)}_{\pi,\IA}&0&0& 0& q+1& 0\\
1&\frac{1}{2}(q^2-q)&\Lambda^{(1)}_{\IA,\pi}&0&0& 0& 1& 0\\\hline

1&\IA\t{-lines}&\Pi^{(-1)}_{\pi,\IA}&0&0&\frac{1}{3}(q+1)& 0&\frac{2}{3}(q+1)\\
1&\frac{1}{2}(q^2-q)&\Lambda^{(-1)}_{\IA,\pi}&0&0& 1& 0& 1\\\hline

2&\EG\t{-lines}&\Pi^{(1)}_{\pi,\EG}&1&1& \frac{1}{6}(q-4)& \frac{1}{2}q&\frac{1}{3} (q-1)\\
1&q^3-q&\Lambda^{(1)}_{\EG,\pi}&q^2-q&q-1& q-4& q& q-1\\\hline

2&\EG\t{-lines}&\Pi^{(-1)}_{\pi,\EG}&1&1&\frac{1}{6}(q-2)& \frac{1}{2}(q-2)&\frac{1}{3}(q+1)\\
1&q^3-q&\Lambda^{(-1)}_{\EG,\pi}&q^2-q&q-1& q-2& q-2& q+1\\\hline

\ge2&\EnG\t{-lines}&\Pi^{(1)}_{\pi,\EnG}&0&1& \frac{q^2-3q+4}{6(q-1)}& \frac{q^2-q-2}{2(q-1)}& \frac{q^2+1}{3(q-1)}\\
\ge2&\scriptstyle{q^4-q^3-q^2+q}&\Lambda^{(1)}_{\EnG,\pi}&0&(q-1)^2&\scriptstyle{q^2-3q+4}&\scriptstyle{q^2-q-2}&q^2+1\\\hline

\ge2&\EnG\t{-lines}&\Pi^{(-1)}_{\pi,\EnG}&0&1& \frac{1}{6}(q-2)& \frac{1}{2}q& \frac{1}{3}(q+1)\\
\ge2&\scriptstyle{q^4-q^3-q^2+q}     &\Lambda^{(-1)}_{\EnG,\pi}&0&(q-1)^2&\scriptstyle{q^2-3q+2}&q^2-q& q^2-1\\\hline

\ge2&\EnG\t{-lines}&\Pi^{(0)}_{\pi,\EnG}&0&1&\frac{q^2-3q+3}{6(q-1)}&\frac{q^2-q-1}{2(q-1)}&\frac{q^2}{3(q-1)}\\
&\scriptstyle{q^4-q^3-q^2+q}&\Lambda^{(0)}_{\EnG,\pi}&0&(q-1)^2&\scriptstyle{q^2-3q+3}&\scriptstyle{q^2- q-1}&q^2\\\hline
\end{array}
$
\label{tab1}
\end{table}

\newpage
Table 1: continue\\
Values $\Pi_{\pi,\lambda}^{(\xi)}$ (top entry) and $\Lambda_{\lambda,\pi}^{(\xi)}$ (bottom entry) for submatrices $\I_{\pi,\lambda}^{\Pi\Lambda}$ of the plane-line incidence matrix of $\PG(3,q)$, $q\equiv\xi\pmod3$, $\xi\in\{1,-1,0\}$, $q\ge5$, $\pi\in\Pk$, $\lambda\in\Lk^{(\ne0)}\cup\Lk^{(0)}$. The superscript~$(\xi)$ is omitted  if the values $\Pi_{\pi,\lambda}^{(\xi)}$ and $\Lambda_{\lambda,\pi}^{(\xi)}$ are the same for all~$q$

$
\begin{array}
{@{}c@{~\,}ccccccc@{}}\hline
&\O_\lambda&\N_\pi\rightarrow&\Gamma\t{-}&2_\C\t{-}&3_\C\t{-}&\overline{1_\C}\t{-}\vphantom{H^{H^{H^H}}}&0_\C\t{-}\\
L_{\lambda\Sigma}^{\mathrm{od}}&\downarrow&&\,\t{planes}&\t{planes}&\t{planes}&\t{planes}&\t{planes}\\
L_{\lambda\Sigma}^{\mathrm{ev}}&\lambda\t{-lines}&&q+1&q^2+q&\frac{1}{6}(q^3-q)& \frac{1}{2}(q^3-q)&\frac{1}{3}(q^3-q)\vphantom{H_{H_H}}\\\hline
1&\Ar\t{-lines}&\Pi_{\pi,\Ar}^{(0)}&q+1&0&0&0&0\\
&1&\Lambda^{(0)}_{\Ar,\pi}&1&0&0&0&0\\\hline

3&\EA\t{-lines}&\Pi_{\pi,\EA}^{(0)}&1&\frac{q}{q+1}&\frac{q(q-2)}{6(q+1)}&\frac{q^2}{2(q+1)}&\frac{1}{3}q\\
&\scriptstyle{q^3+q^2-q-1}&\Lambda^{(0)}_{\EA,\pi}&q^2-1&q-1&q-2&q&q+1\\\hline
\end{array}
$
\newpage

\begin{table}[htbp]
\caption{Values $\Pi_{\pi,\lambda_j}^{(\xi)\bullet}$ (top entry) and
 $\Lambda_{\pi,\lambda_j}^{(\xi)\bullet}$ (bottom entry) for submatrices $\I_{\pi,\lambda_j}^{\Pi\Lambda}$ of the plane-line incidence matrix of $\PG(3,q), q\ge5,\pi\in\Pk$; even $q\not\equiv0\pmod3$ for $\lambda=\UG$; odd $q$ for $\lambda=\UnG$; odd $q\equiv\xi\pmod3$, $\xi\in\{1,-1\}$, for $\lambda=\EG$; $q\equiv0\pmod3$ for $\lambda=\EA$}
$
\begin{array}
{@{}cccccccc@{}}\hline
\O_{\lambda_j}&\N_\pi\rightarrow&\Gamma\t{-}&2_\C\t{-}&3_\C\t{-}&\overline{1_\C}\t{-}\vphantom{H^{H^{H^H}}}&0_\C\t{-}\\
\downarrow&&\t{planes}&\t{planes}&\t{planes}&\t{planes}&\t{planes}\\
\lambda_j\t{-lines}&&q+1&q^2+q&\frac{1}{6}(q^3-q)& \frac{1}{2}(q^3-q)&\frac{1}{3}(q^3-q)\vphantom{H_{H_{H_{H_H}}}}\\\hline

\UG_1\t{-lines}&\Pi_{\pi,\UG_1}^{(\ne0)\mathrm{ev}}&1&q& 0& 0&0\\
q+1&\Lambda_{\UG_1,\pi}^{(\ne0)\mathrm{ev}}&1&1& 0& 0&0\\\hline

\UG_2\t{-lines}&\Pi_{\pi,\UG_2}^{(\ne0)\mathrm{ev}}&1&0& \frac{1}{2}q&\frac{1}{2}q&0\\
q^2-1&\Lambda_{\UG_2,\pi}^{(\ne0)\mathrm{ev}}&q-1&0& 3& 1&0\\\hline\hline

\UnG_1\t{-lines}&\Pi_{\pi,\UnG{_1}}^\mathrm{od}&0&1& \frac{1}{2}(q-1)& \frac{1}{2}(q+1)&0\\
\frac{1}{2}(q^3-q)&\Lambda_{\UnG{_1},\pi}^\mathrm{od}&0&\frac{1}{2}(q-1)& \frac{3}{2}(q-1)& \frac{1}{2}(q+1)&0\\\hline

\UnG_2\t{-lines}&\Pi_{\pi,\UnG{_2}}^\mathrm{od}&0&3& \frac{1}{2}(q-3)&\frac{1}{2}( q-1)&0\\
\frac{1}{2}(q^3-q)&\Lambda_{\UnG{_2},\pi}^\mathrm{od}&0&\frac{3}{2}(q-1)& \frac{3}{2}(q-3)& \frac{1}{2}(q-1)&0\\\hline\hline

\EG_1\t{-lines}&\Pi^{(1)\mathrm{od}}_{\pi,\EG_1}&1&0& \frac{1}{6}(q-1)& \frac{1}{2}(q+1)& \frac{1}{3}(q-1)\\
\frac{1}{2}(q^3-q)&\Lambda^{(1)\mathrm{od}}_{\EG_1,\pi}&\frac{1}{2}(q^2-q)&0
&\frac{1}{2}(q-1)&\frac{1}{2}(q+1)&\frac{1}{2}(q-1)\\\hline

\EG_2\t{-lines}&\Pi^{(1)\mathrm{od}}_{\pi,\EG_2}&1&2&\frac{1}{6}(q-7)&\frac{1}{2}(q-1)&\frac{1}{3}(q-1)\\
\frac{1}{2}(q^3-q)&\Lambda^{(1)\mathrm{od}}_{\EG_2,\pi}&\frac{1}{2}(q^2-q)&q-1
&\frac{1}{2}(q-7)&\frac{1}{2}(q-1)&\frac{1}{2}(q-1)\\\hline\hline

\EG_1\t{-lines}&\Pi^{(-1)\mathrm{od}}_{\pi,\EG_1}&1&0& \frac{1}{6}(q+1)& \frac{1}{2}(q-1)& \frac{1}{3}(q+1)\\
\frac{1}{2}(q^3-q)&\Lambda^{(-1)\mathrm{od}}_{\EG_1,\pi}&\frac{1}{2}(q^2-q)&0
&\frac{1}{2}(q+1)&\frac{1}{2}(q-1)&\frac{1}{2}(q+1)\\\hline

\EG_2\t{-lines}&\Pi^{(-1)\mathrm{od}}_{\pi,\EG_2}&1&2&\frac{1}{6}(q-5)&\frac{1}{2}(q-3)& \frac{1}{3}(q+1)\\
\frac{1}{2}(q^3-q)&\Lambda^{(-1)\mathrm{od}}_{\EG_2,\pi}&\frac{1}{2}(q^2-q)&q-1
&\frac{1}{2}(q-5)&\frac{1}{2}(q-3)& \frac{1}{2}(q+1)\\\hline\hline

\EA_1\t{-lines}&\Pi^{(0)\mathrm{od}}_{\pi,\EA_1}&1&1& \frac{1}{6}(q-3)& \frac{1}{2}(q-1)&\frac{1}{3} q\\
q^3-q&\Lambda^{(0)\mathrm{od}}_{\EA_1,\pi}&q^2-q&q-1& q-3& q-1& q\\\hline

\EA_2\t{-lines}&\Pi^{(0)\mathrm{od}}_{\pi,\EA_2}&1&0& \frac{1}{3}q&0& \frac{2}{3}q\\
\frac{1}{2}(q^2-1)&\Lambda^{(0)\mathrm{od}}_{\EA_2,\pi}&\frac{1}{2}(q-1)&0& 1& 0& 1\\\hline

\EA_3\t{-lines}&\Pi^{(0)\mathrm{od}}_{\pi,\EA_3}&1&0& 0& q& 0\\
\frac{1}{2}(q^2-1)&\Lambda^{(0)\mathrm{od}}_{\EA_3,\pi}&\frac{1}{2}(q-1)&0& 0& 1& 0\\\hline
\end{array}
$
\label{tabUG}
\end{table}

The 1-st column of Table \ref{tab1} shows the total number of orbits of $\lambda$-lines  for $q$ odd ($L_{\lambda\Sigma}^{\mathrm{od}}$, top entry) and $q$ even ($L_{\lambda\Sigma}^{\mathrm{ev}}$, bottom entry).

In Table \ref{tabUG}, the values  $\Pi_{\pi,\lambda_j}^{(\xi)\bullet}$ and $\Lambda_{\pi,\lambda_j}^{(\xi)\bullet}$ are given for the following cases:
even $q\not\equiv0\pmod3$ with $\lambda=\UG$; odd $q$ with $\lambda=\UnG$; odd $q\equiv\xi\pmod3$, $\xi\in\{1,-1\}$, with $\lambda=\EG$; and $q\equiv0\pmod3$ with $\lambda=\EA$.

\begin{thm}\label{th3_main_res}
Let $q\ge5$, $q\equiv\xi\pmod3$. Let notations be as in Section \emph{\ref{sec_prelimin}} and  Notations~\emph{\ref{notation_1}, \ref{notation_2}}. The following holds:
\begin{description}
  \item[(i)] In $\PG(3,q)$, for the submatrices $\I^{\Pi\Lambda}_{\pi,\lambda}$ of the plane-line incidence matrix $\I^{\Pi\Lambda}$, the values $\Pi_{\pi,\lambda}^{(\xi)}$ \emph{(}i.e. the exact or average number of $\pi$-planes through a $\lambda$-line\emph{)} and $\Lambda_{\lambda,\pi}^{(\xi)}$ \emph{(}i.e. the exact number of  $\lambda$-lines in a $\pi$-plane\emph{)} are given in Table~\emph{\ref{tab1}}. The numbers $L_{\lambda\Sigma}^{(\xi)\bullet}$ of line orbits under $G_q$ in classes $\O_\lambda$ are also collected in the tables. For the submatrices $\I^{\Pi\Lambda}_{\pi,\lambda_j}$ corresponding to each of two orbits of the classes $\O_4=\O_\UG$, $\O_5=\O_\UnG $, $\O'_5=\O_\EG$, and three orbits of the class $\O_8=\O_\EA$, the values $\Pi_{\pi,\lambda_j}^{(\xi)\bullet}$, $\Lambda_{\pi,\lambda_j}^{(\xi)\bullet}$ are given in Table~\emph{\ref{tabUG}}.

  \item[(ii)]
  Let a class $\O_\lambda$ consist  of a single
  orbit according to Theorem \emph{\ref{th2:DMP_Orb}(i)}. Then, in Table~\emph{\ref{tab1}},  the value of $\Pi_{\pi,\lambda}^{(\xi)}$, $\pi\in\Pk$, is the \emph{exact number} of $\pi$-planes through every $\lambda$-line.

  \item[(iii)]
 Let $\pi\in\Pk$ for all $q$. Let a class $\O_\lambda$ consist of a single orbit according to Theorem \emph{\ref{th2:DMP_Orb}(i)}.   Then the submatrix $\I^{\Pi\Lambda}_{\pi,\lambda}$ of  $\I^{\Pi\Lambda}$ is
  a $(v_r,b_k)$ configuration of Definition \emph{\ref{def2_config}} with $v=\#\N_\pi$, $b=\#\O_\lambda$, $r=\Lambda_{\lambda,\pi}$, $k=\Pi_{\pi,\lambda}$. Also, up to rearrangement of rows and columns, the submatrices $\I^{\Pi\Lambda}_{\pi,\lambda}$ with $\Lambda_{\lambda,\pi}^{(\xi)}=1$ can be viewed as a concatenation of $\Pi_{\pi,\lambda}^{(\xi)}$ identity matrices of order $\#\O_\lambda$. The same holds for the submatrices $\I^{\Pi\Lambda}_{\pi,\lambda_j}$.

    \item[(iv)]
Let $\lambda\in\{\UG,\EG\}$ if $q\not\equiv0\pmod3$, $\lambda\in\{\UG,\EA\}$ if $q\equiv0\pmod3$. Then,  independently of the number of orbits in the class $\O_\lambda$, we have exactly one $\Gamma$-plane through every $\lambda$-line. Up to rearrangement of rows and columns, the submatrices $\I^{\Pi\Lambda}_{\Gamma,\lambda}$ can be viewed as a vertical concatenation of\/ $\Lambda_{\lambda,\Gamma}^{(\xi)}$ identity matrices of order $\#\N_\Gamma=q+1$.

    \item[(v)]
  For $q\not\equiv0\pmod3$, the submatrix $\I^{\Pi\Lambda}_{\Gamma,\RA}$ of $\I^{\Pi\Lambda}$ is a simple complete $2\t{-}(q+1,2,1)$ design in the sense of~\emph{\cite[Section 1.6]{HandbCombDes2v_k_lamb}}.

    \item[(vi)] For all $q\ge5$, all $q+1$ planes through an imaginary chord are $\overline{1_\C}$-planes forming a pencil. The $\binom{q}{2}(q+1)$-orbit of all  $\overline{1_\C}$-planes can be partitioned into $\binom{q}{2}$ pencils of planes having an imaginary chord as axis. If $q\equiv1\pmod3$, a similar property holds for imaginary axes and $\overline{1_\C}$-planes.
\end{description}
\end{thm}
\newpage
\begin{thm}\label{th3:q=2 3 4}
 Let the types of lines and planes be as in Table~\emph{\ref{tab1}}.
\begin{description}
    \item[(i)] Let $q=2$. The group $G_2\cong\mathbf{S}_3\mathbf{Z}_2^3$ contains $8$ subgroups isomorphic to $PGL(2,2)$ divided into two conjugacy classes. For one of these subgroups, the matrices corresponding to the projectivities of the subgroup assume the form described by \eqref{eq2_M}. For this subgroup (and only for it) the line-plane incidence matrix has the form of  Table~\emph{\ref{tab1}} for $q\equiv-1\pmod3$.
    \item[(ii)] Let $q=3$. The group $G_3\cong\mathbf{S}_4\mathbf{Z}_2^3$ contains $24$ subgroups isomorphic to $PGL(2,3)$ divided into four conjugacy classes. For one of these subgroups, the matrices corresponding to the projectivities of the subgroup assume the form described by \eqref{eq2_M}. For this subgroup (and only for it) the line-plane incidence matrix has the form of  Table~\emph{\ref{tab1}}  for $q\equiv0\pmod3$.
    \item[(iii)] Let $q=4$. The group $G_4\cong\mathbf{S}_5\cong P\Gamma L(2,4)$ contains one subgroup isomorphic to $PGL(2,4)$. The matrices corresponding to the projectivities of this subgroup assume the form described by \eqref{eq2_M} and for this subgroup the line-plane incidence matrix has the form of  Table~\emph{\ref{tab1}} for $q\equiv1\pmod3$.
\end{description}
\end{thm}


\section{Some useful relations}\label{sec:useful}
In this section, we omit the superscripts ``$(\xi)$'', ``od'', and ``'ev' as they are the same for all terms in a formula; in particular, we use $\Lk$ and $L_{\lambda\Sigma}$ instead of $\Lk^{(\xi)}$ and $L_{\lambda\Sigma}^{(\xi)\bullet}$. In further, when relations of this section are applied, we add the superscripts if they are necessary by the context.

\begin{lem}\label{lem4_line&nplane}
The following holds:
\begin{description}
  \item[(i)]
  The number  $\Lambda_{\lambda_j,\pi}$ of lines from an orbit $\O_{\lambda_j}$ in a plane of an orbit $\N_\pi$ is the same for all planes of~$\N_\pi$.

  \item[(ii)]
  The total number  $\Lambda_{\lambda,\pi}$ of lines from an orbit union $\O_\lambda$ in a plane of an orbit $\N_\pi$ is the same for all planes of~$\N_\pi$. We have \begin{align}\label{eq4_class_Lambda}
&  \Lambda_{\lambda,\pi}=\sum_{j=1}^{L_{\lambda\Sigma}} \Lambda_{\lambda_j,\pi}.
  \end{align}

  \item[(iii)] The number $\Pi_{\pi,\lambda_j}$ of planes from an orbit $\N_\pi$ through a line of an orbit $\O_{\lambda_j}$ is the same for all lines of $\O_{\lambda_j}$.

  \item[(iv)] The average
number $\Pi_{\pi,\lambda}$ of planes from an orbit $\N_\pi$ through a line of a union $\O_\lambda$ over all lines of $\O_\lambda$ satisfies the following relations:
\begin{align}\label{eq4:Pi_aver}
&\Lambda_{\lambda,\pi}\cdot\# \N_\pi=\Pi_{\pi,\lambda}\cdot\#\O_\lambda;\db\\
&\Pi_{\pi,\lambda}=\frac{1}{\#\O_{\lambda}}\sum_{j=1}^{L_{\lambda\Sigma}}\left(\Pi_{\pi,\lambda_j}\cdot\#\O_{\lambda_j}\right).\label{eq4:Pi_aver2}
\end{align}

 \item[(v)]
If $L_{\lambda\Sigma}=1$, then $\O_\lambda$ is an orbit and the
number  of planes from $\N_\pi$ through a line of $\O_\lambda$ is
the same for all lines of $\O_\lambda$.
In this case $\Pi_{\pi,\lambda}$ is certainly an integer.
     If $\Pi_{\pi,\lambda}$ is not an integer then the union $\O_\lambda$ contains more than one orbit, i.e. $L_{\lambda\Sigma}\ge2$.
\end{description}
\end{lem}

\begin{proof}
\begin{description}
  \item[(i)]
Consider planes $\pk_1$ and $\pk_2$ of~$\N_\pi$. Denote by $\ell$ a line of $\O_{\lambda_j}$. Let $S(\pk_1)$ and $S(\pk_2)$ be subsets of $\O_{\lambda_j}$ such that $S(\pk_1)=\{\ell\in\O_{\lambda_j}|\ell\in\pk_1\}$, $S(\pk_2)=\{\ell\in\O_{\lambda_j}|\ell\in\pk_2\}$. There exists $\varphi\in G_q$  such that  $\pk_2=\pk_1\varphi$. Clearly, $\varphi$ embeds $S(\pk_1)$ in $S(\pk_2)$, i.e. $S(\pk_1)\varphi\subseteq S(\pk_2)$ and $\#S(\pk_1)\le\#S(\pk_2)$. In the same way, $\varphi^{-1}$ embeds $S(\pk_2)$ in $S(\pk_1)$, i.e.  $\#S(\pk_2)\le\#S(\pk_1)$. Thus,  $\#S(\pk_2)=\#S(\pk_1)$.

  \item[(ii)] For fixed $\lambda$, orbits $\O_{\lambda_j}$ do not intersect each other.

  \item[(iii)] The assertion can be proved similarly to case (i).

  \item[(iv)] The cardinality $C_1$ of the multiset consisting of lines of $\O_\lambda$ in all planes of $\N_\pi$ is equal to $\Lambda_{\lambda,\pi}\cdot\# \N_\pi$. The cardinality $C_2$ of the multiset consisting of planes of $\N_\pi$ through all lines of $\O_\lambda$ is  $\Pi_{\pi,\lambda}\cdot\#\O_\lambda$. Every $C_i$ is the number of ones in the incidence submatrix $\I_{\pi,\lambda}^{\Pi\Lambda}$ of $\I^{\Pi\Lambda}$.  Thus, $C_1=C_2$.

      The assertion \eqref{eq4:Pi_aver2} holds as $\O_\lambda$ is \emph{partitioned} by $L_{\lambda\Sigma}$ orbits $\O_{\lambda_j}$.

  \item[(v)] The assertion follows from the case (iii).\qedhere
 \end{description}
 \end{proof}

\begin{corollary}\label{cor4_=0}
  If  $\Pi_{\pi,\lambda}=0$ then $\Lambda_{\lambda,\pi}=0$ and vice versa.
\end{corollary}

\begin{proof}
    The assertions follow from \eqref{eq4:Pi_aver}.
\end{proof}

 \begin{thm}\label{th4_linenplane}
Let the lines of $\PG(3,q)$ be partitioned under $G_q$ into $\#\Lk$ classes $\O_\lambda$ where every class is a union of orbits of $\lambda$-lines, $\lambda\in\Lk$. Also, let the planes of $\PG(3,q)$ be partitioned under $G_q$ by $\#\Pk$ orbits $\N_\pi$ of $\pi$-planes, $\pi\in\Pk$.
The following holds:
\begin{align}
 &\sum_{\pi\in\Pk} \Pi_{\pi,\lambda}=q+1,~\lambda\t{ is fixed};\label{eq4_planes_through_line_sum}\db\\
&\sum_{\lambda\in\Lk}\Lambda_{\lambda,\pi}=\beta_{2,q}=q^2+q+1,~\pi\t{ is fixed}.\label{eq4_lines_in_plane_sum}
\end{align}
 \end{thm}
\begin{proof}
         Relations \eqref{eq4_planes_through_line_sum} and \eqref{eq4_lines_in_plane_sum} hold as the lines and the planes of  $\PG(3,q)$ are \emph{partitioned}  under $G_q$ by unions of line orbits and  by orbits of planes, respectively. In total, in $\PG(3,q)$, there are $q+1$ planes through every line and $\beta_{2,q}$ lines in every plane.
\end{proof}

\begin{thm}\label{th4_ext}
 Let $\ell_\t{\emph{ext}}$ be an external line with respect to $\C$. Let
$\Pi_{\pi}(\ell_\t{\emph{ext}})$ be the number of $\pi$-planes through  $\ell_\t{\emph{ext}}$, $\pi\in\Pk.$ The following holds:
\begin{align} &\Pi_{\Gamma}(\ell_\t{\emph{ext}})+\Pi_{\overline{1_\C}}(\ell_\t{\emph{ext}})+2\Pi_{2_\C}(\ell_\t{\emph{ext}})
+3\Pi_{3_\C}(\ell_\t{\emph{ext}})=q+1;\label{eq4_ext_line2}\db \\
 &\Pi_{0_\C}(\ell_\t{\emph{ext}})=\Pi_{2_\C}(\ell_\t{\emph{ext}})+2\Pi_{3_\C}(\ell_\t{\emph{ext}}).\label{eq4_ext_line3}
\end{align}
\end{thm}

\begin{proof} For \eqref{eq4_ext_line2}, we consider $q+1$ planes through $\ell_\t{ext}$ and a point of $\C$. These planes cannot be $0_\C$-planes. Also, every $2_\C$- and $3_\C$-plane appears two and three times, respectively. Finally, \eqref{eq4_ext_line3} follows from \eqref{eq4_planes_through_line_sum} and \eqref{eq4_ext_line2}.
\end{proof}

\begin{corollary}\label{cor4}
  The following holds:
  \begin{align}\label{eq4_obtainPi1}
  &\Pi_{\pi,\lambda}=\P_{\pi,\lambda}\triangleq\frac{\Lambda_{\lambda,\pi}\cdot\# \N_\pi}{\#\O_\lambda};\db\\
  &\Lambda_{\lambda,\pi}=\L_{\lambda,\pi}\triangleq\frac{\Pi_{\pi,\lambda}\cdot\#\O_\lambda}{\#\N_\pi};\db\label{eq4_obtainLamb1}\\
  &\Pi_{\pi^*,\lambda}=\Ps_{\pi^*,\lambda}\triangleq q+1-\sum_{\pi\in\Pk\setminus\{\pi^*\}} \Pi_{\pi,\lambda},~\lambda\t{ is fixed},~\pi^*\in\Pk;\label{eq4_obtainPi2}\db\\
&\Lambda_{\lambda^*,\pi}=\Ll_{\lambda^*,\pi}\triangleq q^2+q+1-\sum_{\lambda\in\Lk\setminus\{\lambda^*\}}\Lambda_{\lambda,\pi},~\pi\t{ is fixed},,~\lambda^*\in\Lk.\label{eq4_obtainLamb2}
  \end{align}
\end{corollary}

\begin{proof}
The assertions directly follow from \eqref{eq4:Pi_aver}, \eqref{eq4_planes_through_line_sum}, \eqref{eq4_lines_in_plane_sum}.
\end{proof}

\section{The numbers of $\pi$-planes through $\lambda$-lines and $\lambda$-lines in $\pi$-planes, for $\PG(3,q)$, $q\not\equiv0\pmod3$}\label{sec:results_q_ne0}

The values of $\#\N_\pi$, $\#\O_\lambda$, needed for  \eqref{eq4_obtainPi1}, \eqref{eq4_obtainLamb1}, are taken from \eqref{eq2_plane orbit_gen}--\eqref{eq2_classes line q=0mod3}. When we use \eqref{eq4_obtainPi2}, \eqref{eq4_obtainLamb2} the values $\Pi_{\pi,\lambda}$, $\Lambda_{\lambda,\pi}$ obtained above are summed~up.

\begin{thm}\label{th6:RC}For all $q$, the following holds:
\begin{description}
  \item[(i)] An $\RC$-line cannot lie in $\Gamma$-, $\overline{1_\C}$-, and $0_\C$-planes. Thus,
  \begin{align*}
    \Pi_{\pi,\RC}= \Lambda_{\RC,\pi}= 0,~\pi\in\{\Gamma,\overline{1_\C},0_\C\}.
  \end{align*}
  \item[(ii)] The number of\/ $2_\C$-planes and $3_\C$-planes through a real chord of $\C$ is equal to $2$ and $q-1$, respectively. Every $2_\C$-plane \emph{(}resp. $3_\C$-plane\emph{)} contains one \emph{(}resp. three\emph{)}  real chords of $\C$. Thus,
      \begin{align*}
        \Pi_{2_\C,\RC}=2,~ \Pi_{3_\C,\RC}=q-1,~\Lambda_{\RC,2_\C}=1,~\Lambda_{\RC,3_\C}=3.
      \end{align*}
\end{description}
 \end{thm}

\begin{proof}
\begin{description}
  \item[(i)]
  An $\RC$-line contains two points of $\C$; it cannot lie in $\Gamma$-, $\overline{1_\C}$-, $0_\C$-planes as these planes have less than $2$ points in common with the cubic~$\C$.
  \item[(ii)]
 We consider the real chord  through points $K,Q$ of~$\C$.  Every plane through a real chord is either a $2_\C$-plane or a $3_\C$-plane. Each of the $q-1$ points $R$ of $\C\setminus\{K,Q\}$ gives rise to the $3_\C$-plane through $K,Q,R$. Therefore, $\Pi_{3_\C,\RC}=q-1$. By \eqref{eq4_obtainPi2}, $\Pi_{2_\C,\RC}=q+1-\Pi_{3_\C,\RC}=2$.

The assertions on $\Lambda_{\RC,\pi}$ follow from the definitions of the planes. \qedhere
\end{description}
\end{proof}

\begin{thm} \label{th6:Gamma-plane}
\begin{description}
  \item[(i)] For all $q$, a $\Gamma$-plane contains one tangent and $q$ non-tangent unisecants. The tangent and non-tangent unisecants lying in a $\Gamma$-plane do not lie in other $\Gamma$-planes. Thus,
      \begin{align*}
       \Lambda_{\Tr,\Gamma}=1,~\Lambda_{\UG,\Gamma}=q,~\Pi_{\Gamma,\Tr}=\Pi_{\Gamma,\UG}=1.
      \end{align*}

  \item[(ii)] Let $q\not\equiv 0\pmod3$.  Then a $\Gamma$-plane contains
$q$ real axes and $q^2-q$ $\EG$-lines. The  external lines lying in a $\Gamma$-plane do not lie in other $\Gamma$-planes.
Also, there are two $\Gamma$-planes through a real axis of $\Gamma$. Thus,
\begin{align*}
\Lambda^{(\ne0)}_{\RA,\Gamma}=q,~\Lambda^{(\ne0)}_{\EG,\Gamma}=q^2-q,~
\Pi^{(\ne0)}_{\Gamma,\EG}=1,~\Pi^{(\ne0)}_{\Gamma,\RA}=2.
\end{align*}

  \item[(iii)] For all $q$, a $\Gamma$-plane does not contain $\IA$-, $\UnG$-, $\EnG$-, and $\IC$-lines. Thus,
  \begin{align*}
  &\Lambda_{\lambda,\Gamma}=\Pi_{\Gamma ,\lambda}=0,~\lambda\in\{\IA,\UnG,\EnG,\IC\}.
      \end{align*}
 \end{description}
\end{thm}

\begin{proof}
\begin{description}
  \item[(i)] The assertions follow from the definitions of the lines. In total, we have $q+1$ unisecants in the contact point of a $\Gamma$-plane. One of these unisecants is a tangent, the other ones are $\UG$-lines. Also, the intersection of two $\Gamma$-planes is a real axis. By Theorem~\ref{th2_Hirs}(iv), the unisecants and real axes are distinct non-intersecting classes of lines.

  \item[(ii)]  In total there are $q+1$ $\Gamma$-planes. For $q\not\equiv 0\pmod3$, each $\Gamma$-plane intersects the remaining $\Gamma$-planes by distinct lines that provide $q$ real axes in it. Thus, for any $\Gamma$-plane, we have $\Lambda_{\Tr,\Gamma}+\Lambda_{\UG,\Gamma}+\Lambda^{(\ne0)}_{\RA,\Gamma}=1+q+q$. The remaining $\beta_{2,q}-(2q+1)=q^2-q$ lines in the $\Gamma$-plane are $\EG$-lines. The intersection of two $\Gamma$-planes is a real axis; this provides $\Pi^{(\ne0)}_{\Gamma,\EG}=1$ and $\Pi^{(\ne0)}_{\Gamma,\RA}=2$.

  \item[(iii)]  The assertions with respect to $\IA$-, $\UnG$- and $\EnG$-lines follow form the definitions of the lines. Also, if an $\IC$-line lies in a $\Gamma$-plane then the line intersects the tangent belonging to this plane; contradiction, as by Theorem \ref{th2_Hirs}(v) no two chords of $\C$ meet off $\C$. \qedhere
\end{description}
\end{proof}

\begin{thm}\label{th6:imag_chord}
For all $q$, the following holds.
 All $d_\C$-planes, $d=0,2,3$, and all $\Gamma$-planes contain no the imaginary chords.
All $q+1$ planes through an imaginary chord are $\overline{1_\C}$-planes forming a pencil. The $\binom{q}{2}(q+1)$-orbit of all  $\overline{1_\C}$-planes can be partitioned into $\binom{q}{2}$ pencils of planes having an imaginary chord as axis. So,
\begin{align*}
\Lambda_{\IC,\pi}=\Pi_{\pi,\IC}=0,~\pi\in\{\Gamma,2_\C,3_\C,0_\C\},~\Lambda_{\IC,\overline{1_\C}}=1,~\Pi_{\overline{1_\C},\IC}=q+1.
\end{align*}
\end{thm}

\begin{proof}
    Any $2_\C$-plane and $3_\C$-plane contains a real chord. An osculating plane contains a tangent. If a $2_\C$,- or a $3_\C$-, or a $\Gamma$-plane contains an imaginary chord then it intersects the real chord or the tangent; contradiction, as by Theorem \ref{th2_Hirs}(v) no two chords of $\C$ meet off $\C$. Thus, we have a $\overline{1_\C}$-plane through an imaginary chord and any point of $\C$. In total, there are $\#\C=q+1$ such $\overline{1_\C}$-planes for every imaginary chord. Also, by Theorem \ref{th2_Hirs}(iv)(a), $\#\O_\IC=\binom{q}{2}$.
\end{proof}

\begin{thm}\label{th6:0C-UGamma}
For all $q$, a $\UG$- and $\UnG$-line cannot lie in a $0_\C$-plane, i.e.
  $$\Lambda_{\UG,0_\C}=\Pi_{0_\C,\UG}=
  \Lambda_{\UnG,0_\C}=\Pi_{0_\C,\UnG}=0.$$
\end{thm}

\begin{proof} A $\UG$- or $\UnG$-line have one point common with $\C$ whereas a $0_\C$-plane has no such points.
\end{proof}

\begin{thm}\label{th6:tang_unisec}
For all $q$, the following holds:
\begin{description}
  \item[(i)]
We consider a real chord $\R\CC$ and two $\Gamma$-planes in its touch points. Every  $2_\C$-plane through $\R\CC$ intersects one  of these $\Gamma$-planes in its tangent and another  in a non-tangent unisecant. Thus,
\begin{align*}
  & \Lambda_{\Tr,2_\C}=\Lambda_{\UG,2_\C}=1.
  \end{align*}
   Also,
   \begin{align*}
   &\Pi_{2_\C,\UG}=\Pi_{2_\C,\EG}=\Lambda_{\UG,\overline{1_\C}}=1,~\Pi_{\overline{1_\C},\UG}=\Pi_{3_\C,\UG}=
(q-1)/2,\db\\
&\Lambda_{\UG,3_\C}=3,~\Lambda_{\UnG,2_\C}=2(q-1),~\Pi_{2_\C,\UnG}=2,~ \Lambda_{\EG,2_\C}=q-1.
\end{align*}

\item[(ii)]
Through a tangent, there are  $q$  $2_\C$-planes. Also, a tangent cannot lie in a $3_\C$- and $0_\C$-plane; we have no $\overline{1_\C}$-planes through a tangent. Thus,
\begin{align*}
   \Pi_{2_\C,\Tr}=q,~\Lambda_{\Tr,\pi}=\Pi_{\pi,\Tr}=0,~\pi\in\{3_\C,\overline{1_\C},0_\C\}.
\end{align*}
\end{description}
\end{thm}

\begin{proof}
\begin{description}
  \item[(i)] Let $P_1,P_2$ be the intersections points of $\R\CC$ and $\C$ and let $\T_1,\T_2$ and $\Gamma_1,\Gamma_2$ be the corresponding tangents and $\Gamma$-planes. The plane $\pi_1$ through $\T_1$ and $\R\CC$ is a $2_\C$-plane due to Theorem \ref{th2_Hirs}(vi); it intersects $\Gamma_2$ by a unisecant $\U_2$. If $\U_2$ is $\T_2$ then $\T_1$ meets $\T_2$ off $\C$, contradiction, see Theorem \ref{th2_Hirs}(v). Thus, $\U_2$ is a non-tangent unisecant, i.e. an $\UG$-line. Similar case holds for the plane $\pi_2$ through $\T_2$ and $\R\CC$. This gives $\Lambda_{\Tr,2_\C}=\Lambda_{\UG,2_\C}=1$.

      Each of $q$ real chords touched in  $P_1$ gives one $2_\C$-plane through some $\UG$-line also touched in $P_1$; distinct chords give $2_\C$-planes
for distinct $\UG$-lines, i.e. we have $\Pi_{2_\C,\UG}=1$ as in a $\Gamma$-plane there are $q$ $\UG$-lines.

 By Theorem \ref{th2_Hirs}(v), a $3_\C$-plane through $\R\CC$ cannot meet a $\Gamma$-plane in a tangent; it intersects $\Gamma_1$ and $\Gamma_2$ in $\UG$-lines. Clearly, these $\UG$-lines do not coincide with intersection lines of other planes through  $\R\CC$. In total, a $3_\C$-plane contains 3 real chords, see Theorem \ref{th6:RC}(ii); formally this gives (with repetitions) $2\cdot3=6$ $\UG$-lines in a $3_\C$-plane. Each of these six $\UG$-lines is counted twice. So, $\Lambda_{\UG,3_\C}=3.$ By \eqref{eq4_obtainPi1}, we obtain $\Pi_{3_\C,\UG}=\P_{3_\C,\UG}=(q-1)/2$.
    By \eqref{eq4_obtainPi2}, we have $\Pi_{\overline{1_\C},\UG}=\Ps_{\overline{1_\C},\UG}=(q-1)/2$.  Now, by \eqref{eq4_obtainLamb1},  $\Lambda_{\UG,\overline{1_\C}}=\L_{\UG,\overline{1_\C}}=1$.

 We consider the $2_\C$-plane $\pi_1$.
Through each of the two touch points $P_1,P_2$ of $\pi_1$ we have  $q$ unisecants of $\C$ lying in $\pi_1$. By above, one of these $2q$ unisecants is the tangent $\T_1\in\Gamma_1$ whereas another is  an $\UG$-line $\U_2\in\Gamma_2$; the other $2q-2$ unisecants are $\UnG$-lines. So,  $\Lambda_{\UnG,2_\C}=2(q-1)$. Now, by \eqref{eq4_obtainPi1}, we obtain $\Pi_{2_\C,\UnG}=\P_{2_\C,\UnG}=2$.

The $2_\C$-plane $\pi_1$ intersects also all the $q-1$ $\Gamma$-planes of $\Gamma\setminus\{\Gamma_1,\Gamma_2\}$. An intersection line is not a unisecant, or an axis, or a chord. Really, we considered above all the unisecants of $\pi_1$. By  Theorem \ref{th2_Hirs}(v), every plane not in $\Gamma$ contains exactly one axis of~$\Gamma$. As $\pi_1$ contains the tangent $\T_1\in\Gamma_1$, it cannot have another axis. Similarly, if an intersection line is  a chord, it intersects $\T_1$, contradiction, see  Theorem \ref{th2_Hirs}(v).  So, all the $q-1$ intersection lines are external lines in $\Gamma$-planes other than chords. Thus,  $\Lambda_{\EG,2_\C}=q-1$. Now, by  \eqref{eq4_obtainPi1}, we obtain $\Pi_{2_\C,\EG}=\P_{2_\C,\EG}=1$.

  \item[(ii)]
We have $q$ real chords through a point $P$ of $\C$.
Every real chord gives one $2_\C$-plane through the tangent in $P$, see the case (i).  In total, we have $q$ $2_\C$-planes through the tangent.

A tangent intersects $\C$ in one point; it cannot lie in a $0_\C$-plane having no points common with $\C$. Also, by Theorem  \ref{th2_Hirs}(vi), every plane through a tangent meets
$\C$ in at most one point other than the point of contact; therefore, $\Pi_{3_\C,\Tr}=0$.
 Finally, by \eqref{eq4_obtainPi2}, we have $\Pi_{\overline{1_\C},\Tr}=\Ps_{\overline{1_\C},\Tr}$=0.\qedhere
\end{description}
\end{proof}

\begin{thm}\label{th6:2C}
 For all $q$, the following holds:
$$\Lambda_{\RA,2_\C}=\Pi_{2_\C,\RA}=
  \Lambda_{\IA,2_\C}=\Pi_{2_\C,\IA}=0,~\Lambda_{\t{\emph{En}}\Gamma,2_\C}=(q-1)^2,~ \Pi_{2_\C,\t{\emph{En}}\Gamma}=1.$$
\end{thm}

\begin{proof} By Theorem \ref{th2_Hirs}(v), every plane not in $\Gamma$ contains exactly one axis of~$\Gamma$. By Theorem \ref{th6:tang_unisec}(i), $\Lambda_{\Tr,2_\C}=1$, i.e. a $2_\C$ plane contains a tangent. Therefore, a $2_\C$ plane cannot contain $\RA$- and $\IA$-lines. Now,
  by \eqref{eq4_obtainLamb2}, we have
      $\Lambda_{\EnG,2_\C}=\Ll_{\EnG,2_\C}=(q-1)^2.$
      Finally, by \eqref{eq4_obtainPi1}, we obtain $\Pi_{2_\C,\EnG}=\P_{2_\C,\EnG}=1$.
\end{proof}

\begin{thm}\label{th6:LambdaPi_UnG}
    For all $q$, the following holds:
     $$\Lambda_{\UnG,3_\C}=3(q-2),~\Pi_{3_\C,\UnG}=(q-2)/2,~ \Lambda_{\UnG,\overline{1_\C}}=q,~\Pi_{\overline{1_\C},\UnG}=q/2.$$
 \end{thm}

\begin{proof} A $3_\C$-plane contains $q-1$ unisecants in each of the points in common with~$\C$; in total, we have $3(q-1)$ unisecants. By Theorem \ref{th6:tang_unisec}(i), $\Lambda_{\UG,3_\C}=3$, i.e. the three unisecants lie also in $\Gamma$-planes. The remaining unisecants are $\UnG$-lines; we have $\Lambda_{\UnG,3_\C}=3(q-1)-3$. Now, by \eqref{eq4_obtainPi1}, we obtain $\Pi_{3_\C,\UnG}=\P_{3_\C,\UnG}=(q-2)/2$.

A $\overline{1_\C}$-plane contains $q+1$ unisecants. By Theorem \ref{th6:tang_unisec}(i), $\Lambda_{\UG,\overline{1_\C}}=1$, i.e. one unisecant  lies also in a $\Gamma$-plane. The remaining unisecants are $\UnG$-lines; we have $\Lambda_{\UnG,\overline{1_\C}}=q$. Now, by \eqref{eq4_obtainPi1}, we obtain $\Pi_{\overline{1_\C},\UnG}=\P_{\overline{1_\C},\UnG}=q/2$.
\end{proof}

\begin{thm}\label{th6:RA_IA}
    For $q\not\equiv0\pmod3$, the following holds:
    \begin{align*}
   &   \Lambda_{\RA,3_\C}^{(1)}=\Lambda_{\RA,0_\C}^{(1)}=
      \Lambda_{\IA,\overline{1_\C}}^{(1)}=\Lambda_{\IA,3_\C}^{(-1)}=\Lambda_{\IA,0_\C}^{(-1)}=
      \Lambda_{\RA,\overline{1_\C}}^{(-1)}=1,\db\\
      &\Pi_{3_\C,\RA}^{(1)}=(q-1)/3,~\Pi_{0_\C,\RA}^{(1)}=2(q-1)/3,~\Pi_{\overline{1_\C},\IA}^{(1)}=q+1,\dbn\\
      &\Lambda_{\IA,3_\C}^{(1)}=\Pi_{3_\C,\IA}^{(1)}=\Lambda_{\IA,0_\C}^{(1)}=\Pi_{0_\C,\IA}^{(1)}=
      \Lambda_{\RA,\overline{1_\C}}^{(1)}=\Pi_{\overline{1_\C},\RA}^{(1)}=0,\dbn\\
      &\Pi_{3_\C,\IA}^{(-1)}=(q+1)/3,~\Pi_{0_\C,\IA}^{(-1)}=2(q+1)/3,~\Pi_{\overline{1_\C},\RA}^{(-1)}=q-1,\dbn\\
      &\Lambda_{\RA,3_\C}^{(-1)}=\Pi_{3_\C,\RA}^{(-1)}=
      \Lambda_{\RA,0_\C}^{(-1)}=\Pi_{0_\C,\RA}^{(-1)}=
      \Lambda_{\IA,\overline{1_\C}}^{(-1)}
      =\Pi_{\overline{1_\C},\IA}^{(-1)}=0.\nt
    \end{align*}
   \end{thm}

   \begin{proof}
    By Theorem \ref{th2_Hirs}(v), for $q\not\equiv0\pmod3$, every plane not in $\Gamma$ contains exactly one axis of $\Gamma$. By Theorem \ref{th6:tang_unisec}(ii), $\Lambda_{\Tr,3_\C}=0$, i.e. a $3_\C$-plane does not contain a tangent. So, a $3_\C$-plane \emph{must} contain either an $\RA$-line or an $\IA$-line but not together. Thus,  it is sufficient to consider only the following two variants for a $3_\C$-plane:

    (a) $\Lambda_{\RA,3_\C}=1$ and $\Lambda_{\IA,3_\C}=0$;
    (b) $\Lambda_{\RA,3_\C}=0$ and $\Lambda_{\IA,3_\C}=1$.

    (a) Let $\Lambda_{\RA,3_\C}=1$, $\Lambda_{\IA,3_\C}=0$.

    By Corollary \ref{cor4_=0}, we obtain $\Pi_{3_\C,\IA}=0$. By Theorem \ref{th6:2C}, $\Pi_{2_\C,\IA}=0$. As an $\IA$-line is external for $\C$, by \eqref{eq4_ext_line3}, we obtain $\Pi_{0_\C,\IA}=\Pi_{2_\C,\IA}+2\Pi_{3_\C,\IA}=0$ whence $\Lambda_{\IA,0_\C}=0$. By Theorem \ref{th6:Gamma-plane}(iii), $\Pi_{\Gamma,\IA}=0$. Now, by \eqref{eq4_obtainPi2}, we have $\Pi_{\overline{1_\C},\IA}=\Ps_{\overline{1_\C},\IA}=q+1$. By \eqref{eq4_obtainLamb1}, we obtain $\Lambda_{\IA,\overline{1_\C}}=\L_{\IA,\overline{1_\C}}=1$. Therefore, by Theorem \ref{th2_Hirs}(v), it is \emph{necessary} to put $\Lambda_{\RA,\overline{1_\C}}=0$ whence $\Pi_{\overline{1_\C},\RA}=0$.

    By Theorem \ref{th6:tang_unisec}(ii), $\Lambda_{\Tr,0_\C}=0$. Above, we proved $\Lambda_{\IA,0_\C}=0$. So, by Theorem~\ref{th2_Hirs}(v), we \emph{must} put $\Lambda_{\RA,0_\C}=1$. Then, by \eqref{eq4_obtainPi1}, we obtain $\Pi_{0_\C,\RA}=\P_{0_\C,\RA}=2(q-1)/3$. It is an integer if $q\equiv1\pmod3$ whereas for $q\equiv-1\pmod3$ it is not integer.
    By Theorem \ref{th2:DMP_Orb}(i),  the class  $\O_\RA$ is an orbit; therefore, $\Pi_{0_\C,\RA}$ must be integer. Thus, the case (a) is possible only for $q\equiv1\pmod3$.

    By  \eqref{eq4_obtainPi1}, $\Pi_{3_\C,\RA}=\P_{3_\C,\RA}=(q-1)/3$. It is an integer if $q\equiv1\pmod3$.

    (b) Let $\Lambda_{\RA,3_\C}=0$, $\Lambda_{\IA,3_\C}=1$.

    By Corollary \ref{cor4_=0}, we obtain $\Pi_{3_\C,\RA}=0$. By Theorem \ref{th6:2C}, $\Pi_{2_\C,\RA}=0$. So, by \eqref{eq4_ext_line3},  $\Pi_{0_\C,\RA}=0$ whence $\Lambda_{\RA,0_\C}=0$. By Theorem~\ref{th6:Gamma-plane}(ii), $\Pi_{\Gamma,\RA}^{(\ne0)}=2$. Now, by~\eqref{eq4_obtainPi2}, $\Pi_{\overline{1_\C},\RA}=\Ps_{\overline{1_\C},\RA}=q-1$. By \eqref{eq4_obtainLamb1}, we obtain $\Lambda_{\RA,\overline{1_\C}}=\L_{\RA,\overline{1_\C}}=1$. Hence, by Theorem \ref{th2_Hirs}(v), it is \emph{necessary} to put $\Lambda_{\IA,\overline{1_\C}}=0$ whence $\Pi_{\overline{1_\C},\IA}=0$.

    By Theorem \ref{th6:tang_unisec}(ii), $\Lambda_{\Tr,0_\C}=0$. Above, we proved $\Lambda_{\RA,0_\C}=0$. So, by Theorem~\ref{th2_Hirs}(v), we \emph{must} put $\Lambda_{\IA,0_\C}=1$. By \eqref{eq4_obtainPi1}, we have $\Pi_{0_\C,\IA}=\P_{0_\C,\IA}=2(q+1)/3$. It is an integer if $q\equiv-1\pmod3$ but for $q\equiv1\pmod3$ it is not an integer.
    By Theorem \ref{th2:DMP_Orb}(i),  the class  $\O_\IA$ is an orbit; therefore, $\Pi_{0_\C,\IA}$ must be an integer. Thus, the case (b) is possible only for $q\equiv-1\pmod3$. By   \eqref{eq4_obtainPi1}, we obtain $\Pi_{3_\C,\IA}=\P_{3_\C,\IA}=(q+1)/3$. It is an integer if $q\equiv-1\pmod3$.
      \end{proof}

      \begin{corollary}
        Let $q\equiv1\pmod3$.
 Then all $d_\C$-planes with $d=0,2,3$ and all osculating planes contain no the imaginary axes $(\IA$-lines\emph{)}.
All the $q+1$ planes through an $\IA$-line are $\overline{1_\C}$-planes forming a pencil. The $\binom{q}{2}(q+1)$-orbit of all  $\overline{1_\C}$-planes can be partitioned into $\binom{q}{2}$ pencils of planes having an $\IA$-line as axis.
      \end{corollary}

      \begin{proof}
        By above, $\Lambda_{\IA,\pi}^{(1)}=\Pi_{\pi,\IA}^{(1)}=0$, $\pi\in\{\Gamma,2_\C,3_\C,0_\C\}$, $\Lambda_{\IA,\overline{1_\C}}^{(1)}=1$, $\Pi_{\overline{1_\C},\IA}^{(1)}=q+1$. Also,  $\#\O_\IA=\binom{q}{2}$.
      \end{proof}

      \begin{thm}\label{th6:EG_EnG}
    For $q\not\equiv0\pmod3$, the following holds:
    \begin{align*}
   &   \Lambda_{\EG,3_\C}^{(1)}=q-4,~\Lambda_{\EG,\overline{1_\C}}^{(1)}=q,~\Lambda_{\EG,0_\C}^{(1)}=q-1,\db\\
      &\Pi_{3_\C,\EG}^{(1)}=(q-4)/6,~\Pi_{\overline{1_\C},\EG}^{(1)}=q/2,~\Pi_{0_\C,\EG}^{(1)}=(q-1)/3;\dbn\\
   &   \Lambda_{\t{\emph{En}}\Gamma,3_\C}^{(1)}=q^2-3q+4,~
      \Lambda_{\t{\emph{En}}\Gamma,\overline{1_\C}}^{(1)}=q^2-q-2,~\Lambda_{\t{\emph{En}}\Gamma,0_\C}^{(1)}=q^2+1,\dbn\\
      &\Pi_{3_\C,\t{\emph{En}}\Gamma}^{(1)}=\frac{q^2-3q+4}{6(q-1)},
      ~\Pi_{\overline{1_\C},\t{\emph{En}}\Gamma}^{(1)}=\frac{q^2-q-2}{2(q-1)},~
      \Pi_{0_\C,\t{\emph{En}}\Gamma}^{(1)}=\frac{q^2+1}{3(q-1)}.\dbn\\
   &   \Lambda_{\EG,3_\C}^{(-1)}=
      \Lambda_{\EG,\overline{1_\C}}^{(-1)}=q-2,
      ~\Lambda_{\EG,0_\C}^{(-1)}=q+1,\db\\
      &\Pi_{3_\C,\EG}^{(-1)}=(q-2)/6,~\Pi_{\overline{1_\C},\EG}^{(-1)}=(q-2)/2,~\Pi_{0_\C,\EG}^{(-1)}=(q+1)/3;\dbn\\
   &   \Lambda_{\t{\emph{En}}\Gamma,3_\C}^{(-1)}=(q-1)(q-2),~
      \Lambda_{\t{\emph{En}}\Gamma,\overline{1_\C}}^{(-1)}=q^2-q,~\Lambda_{\t{\emph{En}}\Gamma,0_\C}^{(-1)}=q^2-1,\dbn\\
      &\Pi_{3_\C,\t{\emph{En}}\Gamma}^{(-1)}=(q-2)/6,~\Pi_{\overline{1_\C},\t{\emph{En}}\Gamma}^{(-1)}=q/2,~
      \Pi_{0_\C,\t{\emph{En}}\Gamma}^{(-1)}=(q+1)/3.\nt
    \end{align*}
   \end{thm}

\begin{proof}
  Let $q\equiv1\pmod3$. Each  $3_\C$-plane intersects all $q+1$ $\Gamma$-planes. By Theorem~\ref{th6:tang_unisec}(i), $\Lambda_{\UG,3_\C}=3$, i.e.\ the three intersections correspond to unisecants in $\Gamma$-planes. By Theorem~\ref{th6:RA_IA},  $\Lambda_{\RA,3_\C}^{(1)}=1$. An $\RA$-line is the intersection of two $\Gamma$-planes. So, the two  intersections of a $3_\C$-plane with $\Gamma$-planes correspond to real axes. The remaining intersections correspond to external lines; thus, $\Lambda_{\EG,3_\C}^{(1)}=q+1-3-2$. By \eqref{eq4_obtainLamb2}, we obtain $\Lambda_{\EnG,3_\C}^{(1)}=\Ll_{\EnG,3_\C}^{(1)}=q^2-3q+4.$

By Theorem \ref{th6:tang_unisec}(i), $\Pi_{2_\C,\EG}=1$. By \eqref{eq4_ext_line3}, $\Pi_{0_\C,\EG}^{(1)}=(q-1)/3$. By  \eqref{eq4_obtainLamb1}, $\Lambda_{\EG, 0_\C}^{(1)}=\L_{\EG, 0_\C}^{(1)}=q-1$.  By \eqref{eq4_obtainLamb2}, $
    \Lambda_{\EnG,0_\C}^{(1)}=\Ll_{\EnG,0_\C}^{(1)}=q^2+1.$

By Theorem \ref{th6:Gamma-plane}(ii), $\Pi_{\Gamma,\EG}^{(\ne0)}=1$. By \eqref{eq4_obtainPi2}, \eqref{eq4_obtainLamb1}, $\Pi_{\overline{1_\C},\EG}^{(1)}=\Ps_{\overline{1_\C},\EG}^{(1)}=q/2$,  $\Lambda_{\EG, \overline{1_\C}}^{(1)}=\L_{\EG, \overline{1_\C}}^{(1)}=q$. Using \eqref{eq4_obtainLamb2}, we have $
    \Lambda_{\EnG,\overline{1_\C}}^{(1)}=\Ll_{\EnG,\overline{1_\C}}^{(1)}=q^2-q-2.$

  Let $q\equiv-1\pmod3$.
 Each  $3_\C$-plane intersects all $q+1$ $\Gamma$-planes. By Theorem \ref{th6:tang_unisec}(i), $\Lambda_{\UG,3_\C}=3$, i.e.\ the three  intersections correspond to unisecants in $\Gamma$-planes.  The remaining intersections correspond to external lines; so, $\Lambda_{\EG,3_\C}^{(-1)}=q+1-3$. Now,
  using~\eqref{eq4_obtainLamb2}, we obtain $\Lambda_{\EnG,3_\C}^{(-1)}=\Ll_{\EnG,3_\C}^{(-1)}=(q-1)(q-2).$

By Theorem \ref{th6:tang_unisec}(i), $\Pi_{2_\C,\EG}=1$.  By \eqref{eq4_ext_line3}, $\Pi_{0_\C,\EG}^{(-1)}=(q+1)/3$. By  \eqref{eq4_obtainLamb1}, $\Lambda_{\EG, 0_\C}^{(-1)}=\L_{\EG, 0_\C}^{(-1)}=q+1$.  By \eqref{eq4_obtainLamb2},  $\Lambda_{\EnG,0_\C}^{(-1)}=\Ll_{\EnG,0_\C}^{(-1)}=q^2-1.$

By Theorem \ref{th6:Gamma-plane}(ii), $\Pi_{\Gamma,\EG}^{(\ne0)}=1$. By \eqref{eq4_obtainPi2}, $\Pi_{\overline{1_\C},\EG}^{(-1)}=\Ps_{\overline{1_\C},\EG}^{(-1)}=(q-2)/2$. By \eqref{eq4_obtainLamb1}, $\Lambda_{\EG, \overline{1_\C}}^{(-1)}=\L_{\EG, \overline{1_\C}}^{(-1)}=q-2$. By \eqref{eq4_obtainLamb2}, $\Lambda_{\EnG,\overline{1_\C}}^{(-1)}=\Ll_{\EnG,\overline{1_\C}}^{(-1)}=q^2-q.$

Finally, by \eqref{eq4_obtainPi1}, we obtain
$\Pi_{3_\C,\EG}^{(1)}$, $\Pi_{3_\C,\EnG}^{(1)}$, $\Pi_{0_\C,\EnG}^{(1)}$, $\Pi_{\overline{1_\C},\EnG}^{(1)}$, $\Pi_{3_\C,\EG}^{(-1)}$, $\Pi_{3_\C,\EnG}^{(-1)}$, $\Pi_{0_\C,\EnG}^{(-1)}$, and $\Pi_{\overline{1_\C},\EnG}^{(-1)}$, using the values of $\Lambda_{\lambda, \pi}^{(\xi)}$ obtained above.
\end{proof}

\begin{corollary}
For $q\equiv1\pmod3$ and odd $q\equiv-1\pmod3$, the class $\O_6=\O_\EnG=\{\EnG$-lines\} contains at least two line orbits under $G_q$.
\end{corollary}

\begin{proof}
 By Theorem \ref{th6:EG_EnG}, for $q\equiv1\pmod3$,  $\Pi_{0_\C,\EnG}^{(1)}=(q^2+1)/3(q-1)$; it is not an integer as $q^2+1\equiv2\pmod3$ but $3(q-1)\equiv0\pmod9$.  For $q\equiv-1\pmod3$, $\Pi_{\overline{1_\C},\EnG}^{(-1)}=q/2$;  it is not an integer  for odd $q$.  Now we use Lemma~\ref{lem4_line&nplane}(v).
\end{proof}

\begin{thm}\label{th6_design}
  For $q\not\equiv0\pmod3$, the submatrix $\I^{\Pi\Lambda}_{\Gamma,\RA}$ of $\I^{\Pi\Lambda}$ is a simple complete $2\t{-}(q+1,2,1)$ design in the sense of~\emph{\cite[Section 1.6]{HandbCombDes2v_k_lamb}}.
\end{thm}

\begin{proof}
  Any two $\Gamma$-planes intersect each other in an $\RA$-line. All these intersections correspond to $\I^{\Pi\Lambda}_{\Gamma,\RA}$.
\end{proof}

\section{The numbers of $\pi$-planes through $\lambda$-lines and of $\lambda$-lines in $\pi$-planes,  for $\PG(3,q)$, $q\equiv0\pmod3$}\label{sec:esults_q=0}
For $\RC$-, $\Tr$-, $\IC$-, $\UG$-, and $\UnG$-lines we use the results of Section \ref{sec:results_q_ne0}, see Table \ref{tab1} where these results are written without superscripts. Also, we may use the values $ \Lambda_{\EnG,\Gamma}= \Pi_{\Gamma,\EnG}=0$, see Theorem \ref{th6:Gamma-plane}(iii).

\begin{thm} \label{th7:A-line}
 For $q\equiv0\pmod3$, the following holds:
 \begin{align*}
 &\Lambda_{\Ar,\Gamma}^{(0)}=1,~\Pi_{\Gamma,\Ar}^{(0)}=q+1,~\Lambda_{\Ar,\pi}^{(0)}=\Pi_{\pi,\Ar}^{(0)}=0,~\pi\in\{2_\C,3_\C,\overline{1_\C},0_\C\},\db\\
 & \Lambda_{\EA,\Gamma}^{(0)}=q^2-1,~ \Pi_{\Gamma,\EA}^{(0)}=1.
 \end{align*}
\end{thm}

\begin{proof} For $q\equiv0\pmod3$, $\Gamma$-planes form a pencil with axis $\Ar$-line, see Section~\ref{sec_prelimin}. This implies the first row of the assertion.
 By~\eqref{eq4_obtainLamb2}, we obtain  $\Lambda_{\EA,\Gamma}^{(0)}=\Ll_{\EA,\Gamma}^{(0)}=q^2-1.$   By   \eqref{eq4_obtainPi1},  $\Pi_{\Gamma,\EA}^{(0)}=\P_{\Gamma,\EA}^{(0)}=1$.
\end{proof}

\begin{thm}\label{th7:2C}
 For $q\equiv0\pmod3$, we have
 \begin{align*}
   \Lambda_{\EA,2_\C}^{(0)}=q-1,~\Pi_{2_\C,\EA}^{(0)}=\frac{q}{q+1},~ \Lambda_{\EnG,2_\C}^{(0)}=(q-1)^2,~ \Pi_{2_\C,\EnG}^{(0)}=1.
\end{align*}
\end{thm}

\begin{proof} In the  proof of Theorem \ref{th6:tang_unisec}(i),  it is shown that a $2_\C$-plane intersects two $\Gamma$-planes, placed in its points in common with $\C$, by unisecants to $\C$ and the other $q-1$ $\Gamma$-planes by lines external with respect to $\C$. As these external lines lie in $\Gamma$-planes, they intersect the axis ($\Ar$-line), i.e. they are $\EA$-lines. Thus, $ \Lambda_{\EA,2_\C}^{(0)}=q-1$. Now, using \eqref{eq4_obtainLamb2}, we obtain $\Lambda_{\EnG,2_\C}^{(0)}=\Ll_{\EnG,2_\C}^{(0)}=(q-1)^2.$
      Finally, by~\eqref{eq4_obtainPi1}, we have $\Pi_{2_\C,\EA}^{(0)}=\P_{2_\C,\EA}^{(0)}=q/(q+1)$, $\Pi_{2_\C,\EnG}^{(0)}=\P_{2_\C,\EnG}^{(0)}=1$.
\end{proof}

\begin{thm}\label{th7:EA_EnG}
 For $q\equiv0\pmod3$, the following holds:
 \begin{align*}
 &\Lambda_{\EA,3_\C}^{(0)}=q-2,~\Lambda_{\EA,\overline{1_\C}}^{(0)}=q,~\Lambda_{\EA,0_\C}^{(0)}=q+1,\db\\
 &\Lambda_{\EnG,3_\C}^{(0)}=q^2-3q+3,~\Lambda_{\EnG,\overline{1_\C}}^{(0)}=q^2-q-1,~\Lambda_{\EnG,0_\C}^{(0)}=q^2,\db\\
 &\Pi_{3_\C,\EA}^{(0)}=\frac{q(q-2)}{6(q+1)},~ \Pi_{\overline{1_\C},\EA}^{(0)}=\frac{q^2}{2(q+1)},~ \Pi_{0_\C,\EA}^{(0)}=\frac{1}{3}q, \db\\
 & \Pi_{3_\C,\EnG}^{(0)}=\frac{q^2-3q+3}{6(q-1)},~ \Pi_{\overline{1_\C},\EnG}^{(0)}=\frac{q^2-q-1}{2(q-1)},~ \Pi_{0_\C,\EnG}^{(0)}=\frac{q^2}{3(q-1)}.
  \end{align*}
\end{thm}

\begin{proof}
A $3_\C$-plane intersects all $q+1$ $\Gamma$-planes. Exactly three of these intersections are unisecants of $\C$ as $\Lambda_{\UG,3_\C}=3$, see Theorem \ref{th6:tang_unisec}(i). The other $q-2$ intersections correspond to lines external with respect to $\C$. As these external lines lie in $\Gamma$-planes they intersect the axis ($\Ar$-line), i.e. they are $\EA$-lines. So, $\Lambda_{\EA,3_\C}^{(0)}=q-2$. Similarly, a $\overline{1_\C}$-plane intersects exactly one $\Gamma$-plane by a unisecant, see $\Lambda_{\UG,\overline{1_\C}}=1$ in Theorem \ref{th6:tang_unisec}(i); the intersections with the other $q$ $\Gamma$-planes provide $\Lambda_{\EA,\overline{1_\C}}^{(0)}=q$. Finally, all $q+1$ intersections of a $0_\C$-plane with $\Gamma$-planes are external lines by the definition of a $0_\C$-plane. This gives $\Lambda_{\EA,0_\C}^{(0)}=q+1$.

Now, using \eqref{eq4_obtainLamb2}, we obtain $\Lambda_{\EnG,3_\C}^{(0)}$, $\Lambda_{\EnG,\overline{1_\C}}^{(0)}$, and  $\Lambda_{\EnG,0_\C}^{(0)}$.

Finally, by  \eqref{eq4_obtainPi1}, we obtain $\Pi_{3_\C,\EA}^{(0)}$, $\Pi_{\overline{1_\C},\EA}^{(0)}$, $\Pi_{0_\C,\EA}^{(0)}$, $\Pi_{3_\C,\EnG}^{(0)}$,  $\Pi_{\overline{1_\C},\EnG}^{(0)}$, and $\Pi_{0_\C,\EnG}^{(0)}$ , using the values of $\Lambda_{\lambda, \pi}^{(0)}$ obtained above.
\end{proof}

\begin{corollary}
For $q\equiv0\pmod3$, the class $\O_6=\O_\EnG=\{\EnG$-lines\} contains at least two line orbits under $G_q$.
\end{corollary}

\begin{proof}
 By Theorem \ref{th7:EA_EnG},  $\Pi_{\overline{1_\C},\EnG}^{(0)}=(q^2-q-1)/2(q-1)$; it is not an integer as the numerator is odd but the denominator is even.  Now we use Lemma~\ref{lem4_line&nplane}(v).
\end{proof}

\begin{thm}
 Let $\pi\in\Pk$. Let a class $\O_\lambda$ consist of a single orbit.
Then the submatrix $\I^{\Pi\Lambda}_{\pi,\lambda}$ of  $\I^{\Pi\Lambda}$ is
  a $(v_r,b_k)$ configuration of Definition \emph{\ref{def2_config}} with $v=\#\N_\pi$, $b=\#\O_\lambda$, $r=\Lambda_{\lambda,\pi}$, $k=\Pi_{\pi,\lambda}$. Also, up to rearrangement of rows and columns, the submatrices $\I^{\Pi\Lambda}_{\pi,\lambda}$ with $\Lambda_{\lambda,\pi}^{(\xi)}=1$ can be viewed as a concatenation of $\Pi_{\pi,\lambda}^{(\xi)}$ identity matrices of order $\#\O_\lambda$. The same holds for the submatrices $\I^{\Pi\Lambda}_{\pi,\lambda_j}$.
\end{thm}

\begin{proof}
As the  class $\O_\lambda$ is an orbit, $\I^{\Pi\Lambda}_{\pi,\lambda}$ contains $\Pi_{\pi,\lambda}$ (resp. $\Lambda_{\lambda,\pi}$) ones in every row (resp. column), see Lemma~\ref{lem4_line&nplane}. In $\PG(3,q)$, two planes intersect along a line. Therefore, two points of  $\I^{\Pi\Lambda}_{\pi,\lambda}$ are connected by at most one line. If $\Lambda_{\lambda,\pi}^{(\xi)}=1$, $\I^{\Pi\Lambda}_{\pi,\lambda}$ contains $\Pi_{\pi,\lambda}^{(\xi)}$ (resp.\ 1) ones in every row (resp. column).
\end{proof}

\section{The numbers of $\pi$-planes through $\lambda_j$-lines and $\lambda_j$-lines in $\pi$-planes in the orbits forming classes $\O_\UG$, $\O_\UnG$, $\O_\EG$, and $\O_\EA$}\label{sec-split}

\begin{thm}\label{th7:Gamma}
Let $\lambda\in\{\UG,\EG\}$ if $q\not\equiv0\pmod3$; $\lambda\in\{\UG,\EA\}$ if $q\equiv0\pmod3$. Then,  independently of the number of orbits in the class $\O_\lambda$, we have exactly one $\Gamma$-plane through every $\lambda$-line. Moreover, up to rearrangement of rows and columns, the submatrices $\I^{\Pi\Lambda}_{\Gamma,\lambda}$ can be viewed as a vertical concatenation of\/ $\Lambda_{\lambda,\Gamma}^{(\xi)}$ identity matrices of order $\#\N_\Gamma=q+1$.
\end{thm}

\begin{proof}
 By the definitions of the lines, one $\Gamma$-plane through a line always exists. If we have two $\Gamma$-planes through a line then it is an $\RA$-line.
\end{proof}

\begin{corollary}\label{cor7:Gamma}
We consider
  two $\frac{1}{2}(q^3-q)$-orbits of $\EG$-lines, for odd $q\not\equiv0\pmod3$, and three orbits $\O_{\EA_1}$, $\O_{\EA_2}$, and $\O_{\EA_3}$ of $\EA$-lines of sizes $q^3-q$, $\frac{1}{2}(q^2-1)$, and $\frac{1}{2}(q^2-1)$, respectively, for $q\equiv0\pmod3$. The following holds:
  \begin{align}\label{eq7:Gamma}
  & \Pi_{\Gamma,\EG_j}^{(\ne0)\mathrm{od}}=\Pi_{\Gamma,\EA_i}^{(0)\mathrm{od}}=1,~j=1,2,~i=1,2,3;\db\\
  &\Lambda_{\EG_j,\Gamma}^{(\ne0)\mathrm{od}}=\frac{1}{2}(q^2-q),~j=1,2;~\Lambda_{\EA_1,\Gamma}^{(0)\mathrm{od}}=q^2-q,~
  \Lambda_{\EA_i,\Gamma}^{(0)\mathrm{od}}=\frac{1}{2}(q-1),~i=2,3.\nt
  \end{align}
\end{corollary}

\begin{proof}
 We use  Theorem \ref{th2:DMP_Orb}(iv)(v). The 1-st row of \eqref{eq7:Gamma} follows from Theorem~\ref{th7:Gamma}. The values in the 2-nd row are obtained by \eqref{eq4_obtainLamb1}.
\end{proof}

\begin{thm}\label{th6:two_orb_UG}
Let $q$ be even. For the $(q+1)$-orbit $\O_{\UG_1}$ and the $(q^2-1)$-orbit $\O_{\UG_2}$ of the class $\O_4=\O_\UG$, the following holds, see Table \emph{\ref{tabUG}}:
\begin{align*}
 &\Pi_{\Gamma,\UG_1}^{(\ne0)\mathrm{ev}}=\Pi_{\Gamma,\UG_2}^{(\ne0)\mathrm{ev}}=\Lambda_{\UG_1,\Gamma}^{(\ne0)\mathrm{ev}}=\Lambda_{\UG_1,2_\C}^{(\ne0)\mathrm{ev}}
 =\Lambda_{\UG_2,\overline{1_\C}}^{(\ne0)\mathrm{ev}}=1;~\Pi_{2_\C,\UG_1}^{(\ne0)\mathrm{ev}}=q;\db\\
 &\Pi_{\pi,\UG_1}^{(\ne0)\mathrm{ev}}=\Lambda_{\UG_1,\pi}^{(\ne0)\mathrm{ev}}=0\t{ if }\pi\in\{3_\C,\overline{1_\C},0_\C\};
 ~\Pi_{3_\C,\UG_2}^{(\ne0)\mathrm{ev}}=\Pi_{\overline{1_\C},\UG_2}^{(\ne0)\mathrm{ev}}=\frac{1}{2}q;\db\\
 &\Pi_{\pi,\UG_2}^{(\ne0)\mathrm{ev}}=\Lambda_{\UG_2,\pi}^{(\ne0)\mathrm{ev}}=0\t{ if }\pi\in\{2_\C,0_\C\};~\Lambda_{\UG_2,\Gamma}^{(\ne0)\mathrm{ev}}=q-1;~\Lambda_{\UG_2,3_\C}^{(\ne0)\mathrm{ev}}=3.
\end{align*}
\end{thm}

\begin{proof}
 By Theorem \ref{th7:Gamma}, $\Pi_{\Gamma,\UG_j}^{(\ne0)\mathrm{ev}}=1$. Through a $\UG$-line, there are $q+1$ planes one of which is a $\Gamma$-plane. For a line of the $(q+1)$-orbit $\O_{\UG_1}$, the remaining $q$ planes are $2_\C$-planes, see Theorems~\ref{th2_Hirs}(vi) and \ref{th2:DMP_Orb}(ii), So, $\Pi_{2_\C,\UG_1}^{(\ne0)\mathrm{ev}}=q$. Now, by \eqref{eq4_planes_through_line_sum} and Corollary~\ref{cor4_=0}, we obtain $\Pi_{\pi,\UG_1}^{(\ne0)\mathrm{ev}}=\Lambda_{\UG_1,\pi}^{(\ne0)\mathrm{ev}}=0$, $\pi\in\{3_\C,\overline{1_\C},0_\C\}$. By \eqref{eq4_obtainLamb1}, we have $\Lambda_{\UG_1,\Gamma}^{(\ne0)\mathrm{ev}}=
\Lambda_{\UG_1,2_\C}^{(\ne0)\mathrm{ev}}=1$. Now, by \eqref{eq4_class_Lambda}, using $\Lambda_{\UG,\pi}$ and $\Lambda_{\UG_1,\pi}^{(\ne0)\mathrm{ev}}$, we obtain all $\Lambda_{\UG_2,\pi}^{(\ne0)\mathrm{ev}}$ and then, by \eqref{eq4_obtainPi1}, we calculate all $\Pi_{\pi,\UG_2}^{(\ne0)\mathrm{ev}}$.
\end{proof}

Remind that for $q\equiv-1\pmod4$ (resp. $q\equiv1\pmod4$), $-1$ is a non-square (resp. square) in $\F_q$. Also, for $q\equiv-1\pmod3$ (resp. $q\equiv1\pmod3$), $-3$ is a non-square (resp. square) in $\F_q$.

\begin{thm}\label{th6:two_orb_UnG}
Let $q$ be odd. For the $\frac{1}{2}(q^3-q)$-orbits $\O_{\UnG_1}$ and $\O_{\UnG_2}$ of the class $\O_5=\O_\UnG$, the following holds, see Table \emph{\ref{tabUG}}:
\begin{align*}
&\Pi_{\pi,\UnG_j}^{\mathrm{od}}=\Lambda_{\UnG_j,\pi}^{\mathrm{od}}=0,~\pi=\Gamma,0_\C,~ j=1,2;~\Pi_{2_\C,\UnG_1}^{\mathrm{od}}=1;~\Pi_{2_\C,\UnG_2}^{\mathrm{od}}=3;\db\\
&\Lambda_{\UnG_1,2_\C}^{\mathrm{od}}=\Pi_{3_\C,\UnG_1}^{\mathrm{od}}=\Pi_{\overline{1_\C},\UnG_2}^{\mathrm{od}}=\Lambda_{\UnG_2,\overline{1_\C}}^{\mathrm{od}}=
\frac{1}{2}(q-1);~\db\\
&\Pi_{3_\C,\UnG_2}^{\mathrm{od}}=\frac{1}{2}(q-3);~\Lambda_{\UnG_2,3_\C}^{\mathrm{od}}=\frac{3}{2}(q-3);
~\Lambda_{\UnG_2,2_\C}^{\mathrm{od}}=\Lambda_{\UnG_1,3_C}^{\mathrm{od}}=\frac{3}{2}(q-1);\db\\
& \Pi_{\overline{1_\C},\UnG_1}^{\mathrm{od}}=\Lambda_{\UnG_1,\overline{1_\C}}^{\mathrm{od}}=\frac{1}{2}(q+1).
\end{align*}
\end{thm}

\begin{proof}
 By the definition, $\Gamma$- and $0_\C$-planes do no contain $\UnG$-lines. So, $\Pi_{\Gamma,\UnG_j}^{\mathrm{od}}=\Pi_{0_\C,\UnG_j}^{\mathrm{od}}=0$, $j=1,2$.

 Now, see Theorem \ref{th2:DMP_Orb}(iii) and \cite[Theorem 6.13, Proof]{DMP_OrbLine}, we consider a plane $\pk=\boldsymbol{\pi}(c_0,c_1,c_2,c_3)$, $c_i\in\F_q$, through the line $\ell'=\overline{\Pf(0,0,0,1)\Pf(1,0,1,0)}$ of $\O_{\UnG_j}$, $j\in\{1,2\}$. We find the number $\Nb$ of points $P(t)$ in $\pk$ other than $P(0)$, see~\eqref{eq2:P(t)}. If and only if $\Nb=1$, $\pk$ is a $2_\C$-plane.
 By \eqref{eq2_plane}, $c_0+c_2=c_3=0$ whence $\pk=\boldsymbol{\pi}(c_0,c_1,-c_0,0)$.

 If $c_0=0$ then $\Nb=1$, $P(\infty)=\Pf(1,0,0,0)\in\pk$.

Let $c_0\ne0$. Then $P(\infty)\not\in\pk$. If $P(t)=\Pf(t^3,t^2,t,1)\in\pk$, $t\in\F^*_q$, then $c_0t^3+$ $c_1t^2-c_0t=0$ and $t^2+ct-1=0$, $c\in\F_q$, whence $t=-c/2\pm\sqrt{(c/2)^2+1}$. If $q\equiv-1\pmod4$, we have $\Nb\in\{0,2\}$ when $c$ runs over $\F_q$; if $q\equiv1\pmod4$, we have $\Nb=1$ exactly for two values of $c$ with $\sqrt{(c/2)^2+1}=0$.

So, when $c_0,c_1$ runs over $\F_q$, there are either one or three cases $\Nb=1$ that corresponds to $\Pi_{2_\C,\UnG_j}^{\mathrm{od}}\in\{1,3\}$.
By Theorem \ref{th6:tang_unisec}(i), $\Pi_{2_\C,\UnG}=2$, whence, by \eqref{eq4:Pi_aver2} and Theorem \ref{th2:DMP_Orb}(iii), we have
$\Pi_{2_\C,\UnG_1}^{\mathrm{od}}+\Pi_{2_\C,\UnG_2}^{\mathrm{od}}=4.$ Therefore, if $\Pi_{2_\C,\UnG_1}^{\mathrm{od}}=1$ then $\Pi_{2_\C,\UnG_2}^{\mathrm{od}}=3$ and vice versa. We put $\Pi_{2_\C,\UnG_1}^{\mathrm{od}}=1$ w.l.o.g.

Consider $q$ planes through a $\UnG$-line and a point of $\C$.  They are either $2_\C$- or $3_\C$-planes; in that, each $3_\C$-plane appears two times. So, $\Pi_{2_\C,\UnG_j}^{\mathrm{od}}+2\Pi_{3_\C,\UnG_j}^{\mathrm{od}}=q$ whence, by above, $\Pi_{3_\C,\UnG_1}^{\mathrm{od}}=(q-1)/2$. Also, by \eqref{eq4_obtainPi2}, $\Pi_{\overline{1_\C},\UnG_1}^{\mathrm{od}}=(q+1)/2$.
Now, by \eqref{eq4_obtainLamb1}, we obtain all $\Lambda_{\UnG_1,\pi}^{\mathrm{od}}$ from $\Pi_{\pi,\UnG_1}^{\mathrm{od}}$. Then by \eqref{eq4_class_Lambda} and
 \eqref{eq4_obtainPi1}, we calculate all $\Lambda_{\UnG_2,\pi}^{\mathrm{od}}$ and $\Pi_{\pi,\UnG_2}^{\mathrm{od}}$.
\end{proof}

\begin{lem}\label{lem6:Q_Q+1}
 Let $q\equiv\xi\pmod3$ be odd; let also $q\equiv\beta\pmod4$, $\xi,\beta\in\{1,-1\}$. Let $f(x)=-\frac{4}{3}x^2-3$.
 Let $V^{(\xi,\beta)}=\{c\in\F_q^* | f(c) \text{ is a non-square in }\F_q^*\}$, $R^{(\beta)}=\{c\in\F_q^* | f(c)=0\}$. Then $\#R^{(\beta)}=\beta+1$, $\#V^{(\xi,\beta)}=\frac{1}{2}(q+2\xi-2-\beta)$.
\end{lem}

\begin{proof}
The roots of $f(x)$ are $\pm\frac{3}{2}\sqrt{-1}$. This explains $\#R^{(\beta)}$.

Let $\eta$ be the quadratic character of $\F_q$. For $a\in\F_q^*$, $\eta(a)=1$ if $a$ is a square in $\F_q^*$ and $\eta(a)=-1$ otherwise. By \cite[Theorem 5.18]{Lidl_Nied}, $$
 \sum\limits_{c\in\F_q\setminus R^{(\beta)}}\eta(f(c))=-\eta\left(-\frac{4}{3}\right)=-\xi$$
 where by $c\in\F_q\setminus R^{(\beta)}$ we note that $\eta(0)$ is not defined. As the number $q-\#R^{(\beta)}$ of summands in $\sum_c$ is odd,  $(q-\#R^{(\beta)}+1)/2$ summands are equal to $-\xi$ while $(q-\#R^{(\beta)}-1)/2$ ones are $\xi$. Also, $\eta(f(0))=\eta(-3)=\xi$. Now $\#V^{(\xi,\beta)}$ can be obtained by straightforward calculation.
\end{proof}
\begin{thm}\label{th6:two_orb_EG}
Let $q\equiv\xi\pmod3$ be odd, $\xi\in\{1,-1\}$. For the $\frac{1}{2}(q^3-q)$-orbits $\O_{\EG_j}$, $j=1,2$, of the class $\O'_5=\O_\EG$, in addition to Corollary \emph{\ref{cor7:Gamma}} the following holds, see Table~\emph{\ref{tabUG}}:
\begin{align*}
 & \Pi_{2_\C,\EG_1}^{(\xi)\mathrm{od}}=\Lambda_{\EG_1,2_\C}^{(\xi)\mathrm{od}}=0; ~
 \Pi_{2_\C,\EG_2}^{(\xi)\mathrm{od}}=2,~\Lambda_{\EG_2,2_\C}^{(\xi)\mathrm{od}}=q-1;~\Pi_{3_\C,\EG_1}^{(\xi)\mathrm{od}}=\frac{1}{6}(q-\xi); \db\\
 &\Lambda_{\EG_1,3_\C}^{(\xi)\mathrm{od}}=\Lambda_{\EG_1,0_\C}^{(\xi)\mathrm{od}}=\Lambda_{\EG_2,0_\C}^{(\xi)\mathrm{od}}=\frac{1}{2}(q-\xi);~
 \Lambda_{\EG_1,0_\C}^{(\xi)\mathrm{od}}=\Lambda_{\EG_2,0_\C}^{(\xi)\mathrm{od}}=\frac{1}{3}(q-\xi);\db\\
 &\Pi_{\overline{1_\C},\EG_1}^{(\xi)\mathrm{od}}=\Lambda_{\EG_1,\overline{1_\C}}^{(\xi)\mathrm{od}}=\frac{1}{2}(q+\xi);~
\Pi_{\overline{1_\C},\EG_2}^{(\xi)\mathrm{od}}=
 \Lambda_{\EG_2,\overline{1_\C}}^{(\xi)\mathrm{od}}=\frac{1}{2}(q+\xi-2); \db\\
&\Pi_{3_\C,\EG_2}^{(\xi)\mathrm{od}}=\frac{1}{6}(q-\xi-6),~\Lambda_{\EG_2,3_\C}^{(\xi)\mathrm{od}}=\frac{1}{2}(q-\xi-6).
\end{align*}
\end{thm}

\begin{proof}
\looseness -1  We use Theorem \ref{th2:DMP_Orb}(iii)(iv). The null polarity $\A$~\eqref{eq2_null_pol} maps the points $P_0=\Pf(0,0,0,1)$ and $P'=\Pf(1,0,1,0)$
of \cite[Theorem 6.13, Proof]{DMP_OrbLine} to the planes $\pk_0=\boldsymbol{\pi}(1,0,0,0)$ and $\pk'=\boldsymbol{\pi}(0,-3,0,-1)$, respectively. The $\UnG$-line $\ell'=\overline{P_0P'}$ is mapped to an $\EA$-line $\overline{\ell}$ so that $\ell'\A=\pk_0\cap\pk'\triangleq\overline{\ell}$. Let $\overline{\pi}=\boldsymbol{\pi}(c_0,c_1,c_2,c_3)$, $c_i\in\F_q$, be a plane through $\overline{\ell}$. By \cite[Section 15.2]{Hirs_PG3q}, the matrix associated with $\overline{\ell}$ is\\
\centerline{$\widehat{\Lambda}= \left[
\renewcommand\arraystretch{0.85}
         \begin{array}{cccc}
           0 & 0 & 0 & 0 \\
           0 & 0 & 1 & 0 \\
           0 & -1 & 0 & 3 \\
           0 & 0 & -3 & 0 \\
         \end{array}
 \right] $}
and $\overline{\pi}\widehat{\Lambda}=0$. It gives $-c_2=c_1-3c_3=3c_2=0$ whence $\overline{\pi}=\boldsymbol{\pi}(c_0,c_1,0,c_1/3)$.


If $c_0\ne0,c_1=0$ then $\overline{\pi}=\boldsymbol{\pi}(1,0,0,0)=\pk_0=\pi_\t{osc}(0)$,  $P(0)\in\overline{\pi}=\pk_0$. Thus, $\overline{\pi}=\pk_0$ is a $\Gamma$-plane.

If $c_0=0$ then $c_1\ne0$, $\overline{\pi}=\boldsymbol{\pi}(0,1,0,1/3)=\pk'$, $P(\infty)=\Pf(1,0,0,0)\in\overline{\pi}=\pk'$, and $P(t)=\Pf(t^3,t^2,t,1)\in\overline{\pi}=\pk'$ if and only if $t=\pm\sqrt{-1/3}$.  So, $\overline{\pi}=\pk'$ is a $3_\C$-plane for $\xi=1$ and a $\overline{1_\C}$-plane for $\xi=-1$.

Let $c_0\ne0,c_1\ne0$. Then $\overline{\pi}=\boldsymbol{\pi}(1,c,0,c/3)\triangleq\overline{\pi}(c)$, $c\in\F_q^*$, $P(\infty)\notin\overline{\pi}(c)$, $P(t)=\Pf(t^3,t^2,t,1)\in\overline{\pi}(c)$ if and only if $t$ satisfies the equation $F(t)$, see \eqref{eq2_plane}, with the discriminant $\Delta(F)$  obtained by \cite[Lemma 1.18(ii)]{Hirs_PGFF}. We have
\begin{align*}
F(t)=t^3+ct^2+\frac{c}{3}=0,~t\in\F_q,~\Delta(F)=c^2\left(-\frac{4}{3}c^2-3\right),~c\in\F_q^*.
\end{align*}

Let $q\equiv\beta\pmod4$, $\beta\in\{1,-1\}$. Let $\Nb_j^{(\xi,\beta)}$ be the number of values of $c\in\F_q^*$ providing exactly $j$ roots of $F(t)$ in the corresponding $\F_q$.

\emph{The plane $\overline{\pi}(c)$ is a $0_\C$-, $\overline{1_\C}$-, $2_\C$-, and $3_\C$-plane according as $F(t)$ has \emph{0, 1, 2,} and \emph{3} roots in $\F_q$.}

By above, $\Pi_{\Gamma,\EG_j}^{(\xi)\mathrm{od}}=1,~ \Pi_{\overline{1_\C},\EG_j}^{(\xi)\mathrm{od}}=\Nb_1^{(\xi,\beta)}+(1-\xi)/2,~\Pi_{2_\C,\EG_j}^{(\xi)\mathrm{od}}=\Nb_2^{(\xi,\beta)}$, where we take into account the planes $\overline{\pi}=\pk_0$, $\overline{\pi}=\pk'$, and $\overline{\pi}(c)$.

By \cite[Corollary 1.30]{Hirs_PGFF}, where all  $A_i\ne0$ in our case, $F(t)$  has exactly two roots if and only if  $\Delta(F)=0.$ Also, by \cite[Corollary 1.15(ii)]{Hirs_PGFF}, if $\Delta(F)\ne0$ then $F(t)$  has  exactly one root if and only if  $\Delta(F)$ is a non-square in $\F_q$.

We put $j=1$, use  Lemma \ref{lem6:Q_Q+1}, and obtain $\Pi_{\overline{1_\C},\EG_1}^{(\xi)\mathrm{od}}=\#V^{(\xi,\beta)}+(1-\xi)/2=(q+\xi-1-\beta)/2$, $\Pi_{2_\C,\EG_1}^{(\xi)\mathrm{od}}=\#R^{(\beta)}=\beta+1,$
whence, by \eqref{eq4_ext_line2}, \eqref{eq4_ext_line3}, with $\Pi_{\Gamma,\EG_j}^{(\xi)\mathrm{od}}=1$, we have $\Pi_{3_\C,\EG_1}^{(\xi)\mathrm{od}}=(q-\xi-3-3\beta)/6$, $\Pi_{0_\C,\EG_1}^{(\xi)\mathrm{od}}=(q-\xi)/3$. Then, using \eqref{eq4:Pi_aver2} with $\Pi_{\pi,\EG}^{(\xi)\mathrm{od}}$ obtained above, see Table \ref{tab1}, we
obtain $\Pi_{2_\C,\EG_2}^{(\xi)\mathrm{od}}=1-\beta$, $\Pi_{\overline{1_\C},\EG_2}^{(\xi)\mathrm{od}}=(q+\xi-1+\beta)/2$, $\Pi_{3_\C,\EG_2}^{(\xi)\mathrm{od}}=(q-\xi-3+3\beta)/6$, $\Pi_{0_\C,\EG_2}^{(\xi)\mathrm{od}}=(q-\xi)/3$.

For $\beta=-1$, the formulae above give the values $\Pi_{\pi,\EG_j}^{(\xi)\mathrm{od}}$, $j=1,2$, as in Table~\ref{tabUG}. Moreover, $\beta=1$ provides the same values but the numbers $j$ of orbits $\O_{\EG_j}$  change places, i.e. we have $j=2$ instead of $j=1$ and  vice versa. Therefore, $\beta$ does not appear in the final formulae.

In conclusion, by \eqref{eq4_obtainLamb1}, from $\Pi_{\pi,\EG_j}^{(\xi)\mathrm{od}}$ we obtain $\Lambda_{\EG_j,\pi}^{(\xi)\mathrm{od}}$.
\end{proof}

\begin{thm}\label{th7:PiLambEA}
Let $q\equiv0\pmod3$, $q\ge9$. For the orbits $\O_{\EA_1}$, $\O_{\EA_2}$, and $\O_{\EA_3}$ of $\EA$-lines \emph{(}class $\O_8=\O_\EA$\emph{)} of sizes $q^3-q$, $\frac{1}{2}(q^2-1)$, and $\frac{1}{2}(q^2-1)$, respectively, in addition to Corollary \emph{\ref{cor7:Gamma}} the following holds, see Table \emph{\ref{tabUG}}:
\begin{align*}
 &\Pi_{2_\C,\EA_1}^{(0)\mathrm{od}}=\Lambda_{\EA_2,3_\C}^{(0)\mathrm{od}}=
 \Lambda_{\EA_2,0_\C}^{(0)\mathrm{od}}=\Lambda_{\EA_3,\overline{1_\C}}^{(0)\mathrm{od}}=1;~
 \Lambda_{\EA_1,2_\C}^{(0)\mathrm{od}}=\Lambda_{\EA_1,\overline{1_\C}}^{(0)\mathrm{od}}=q-1;\db\\
 & \Pi_{\pi,\EA_j}^{(0)\mathrm{od}}=\Lambda_{\EA_j,\pi}^{(0)\mathrm{od}}=0,~j=2\t{ with }\pi=2_\C,\overline{1_\C},~j=3\t{ with }\pi=2_\C,3_\C,0_\C;\db\\
 &\Pi_{\overline{1_\C},\EA_3}^{(0)\mathrm{od}}=\Lambda_{\EA_1,0_\C}^{(0)\mathrm{od}}=q,~
 \Pi_{3_\C,\EA_2}^{(0)\mathrm{od}}=\Pi_{0_\C,\EA_1}^{(0)\mathrm{od}}=\frac{1}{3}q,~
 \Pi_{0_\C,\EA_2}^{(0)\mathrm{od}}=\frac{2}{3}q;\db\\
 &\Pi_{3_\C,\EA_1}^{(0)\mathrm{od}}=\frac{1}{6}(q-3),~\Lambda_{\EA_1,3_\C}^{(0)\mathrm{od}}=q-3;~
 \Pi_{\overline{1_\C},\EA_1}^{(0)\mathrm{od}}=\frac{1}{2}(q-1).
 \end{align*}
\end{thm}

\begin{proof}
We denote $x_i=\Pi_{2_\C,\EA_i}^{(0)\mathrm{od}}$, $\overline{x}_i=\Lambda_{\EA_i ,2_\C}^{(0)\mathrm{od}}$, $y_i=\Pi_{0_\C,\EA_i}^{(0)\mathrm{od}}$, $\overline{y}_i=\Lambda_{\EA_i ,0_\C}^{(0)\mathrm{od}}$, $i=1,2,3$. Obviously, all the values must be integer.
By Theorems \ref{th7:2C}, \ref{th7:EA_EnG}, $\Pi_{2_\C,\EA}^{(0)}=q/(q+1)$, $\Pi_{0_\C,\EA}^{(0)}=q/3$, whence, by \eqref{eq4:Pi_aver2}, \eqref{eq4_planes_through_line_sum}, \eqref{eq7:Gamma}, we have
\begin{align}\label{eq7:EA2C}
  &qx_1+\frac{1}{2}x_2+\frac{1}{2}x_3=q,~x_i\in\{0,1,\ldots,q\}; \\
   &6qy_1+3y_2+3y_3=2q^2+2q,~y_i\in\{0,1,\ldots,q-x_i\}.\label{eq7:EA0C}
\end{align}

For \eqref{eq7:EA2C}, there are only two solutions $x_1=0,x_2=x_3=q$ and $x_1=1$, $x_2=x_3=0$. Taking into account \eqref{eq4_planes_through_line_sum}, \eqref{eq7:Gamma}, the 1-st solution implies $y_2=y_3=0$, $y_1=(q+1)/3$, contradiction as $y_1$ must be integer. So, $x_1=1,x_2=x_3=0$.

It is easy to see that $y_1=q/3$ is the only possibility to provide $2q^2$ in~\eqref{eq7:EA0C}. Then, by \eqref{eq4_obtainLamb1}, $\overline{y}_1=q$, and by \eqref{eq4_class_Lambda}, $\overline{y}_2+\overline{y}_3=\Lambda_{\EA  ,0_\C}^{(0)\mathrm{od}}-\overline{y}_1=1$, see  Theorem~\ref{th7:EA_EnG}. We put $\overline{y}_2=1$, $\overline{y}_3=0$, w.l.o.g., whence, by \eqref{eq4_obtainPi1}, $y_2=2q/3$, $y_3=0$.

Now, by \eqref{eq4_ext_line3}, we obtain $\Pi_{3_\C,\EA_i}^{(0)\mathrm{od}}=(y_i-x_i)/2$, $i=1,2,3$, and then, by~\eqref{eq4_ext_line2}, we calculate $\Pi_{\overline{1_\C},\EA_i}^{(0)\mathrm{od}}=q+1-\Pi_{\Gamma,\EA_i}^{(0)\mathrm{od}}-2x_i-3\Pi_{3_\C,\EA_i}^{(0)\mathrm{od}}$, $i=1,2,3$. In conclusion, by \eqref{eq4_obtainLamb1}, we obtain the remaining values of $\Lambda_{\EA_i ,\pi}^{(0)\mathrm{od}}$ from $\Pi_{\pi,\EA_i}^{(0)\mathrm{od}}$.
\end{proof}

\section*{Acknowledgments}
 The research of S. Marcugini, and F.~Pambianco was supported in part by the Italian National Group for Algebraic and Geometric Structures and their Applications (GNSAGA - INDAM) and by University of Perugia (Project: Curve, codici e configurazioni di punti, Base Research Fund 2018).

\end{document}